\documentclass[11pt,a4paper]{article}
\usepackage[T1]{fontenc}
\usepackage{lmodern}
\usepackage[margin=25.4mm]{geometry}
\usepackage{amsmath,amssymb,amsthm,mathtools}
\usepackage{microtype,booktabs,array}
\usepackage[numbers,square]{natbib}
\usepackage{xcolor}
\usepackage[colorlinks=true,linkcolor=blue,citecolor=blue,urlcolor=blue]{hyperref}
\hypersetup{pdftitle={On the Brouwer-type Conjecture for Signless Laplacian Eigenvalues of Graphs},pdfauthor={Ruisong Yuan; Xiao-Dong Zhang},pdfsubject={Analytic proof, equality cases, and asymptotic sharpness}}
\allowdisplaybreaks[2]
\numberwithin{equation}{section}
\newtheorem{theorem}{Theorem}[section]
\newtheorem{conjecture}[theorem]{Conjecture}
\newtheorem{lemma}[theorem]{Lemma}
\newtheorem{proposition}[theorem]{Proposition}
\newtheorem{corollary}[theorem]{Corollary}
\theoremstyle{definition}

\theoremstyle{remark}
\newtheorem{remark}[theorem]{Remark}
\newtheorem*{remark*}{Remark}
\DeclareMathOperator{\tr}{tr}
\DeclareMathOperator{\rank}{rank}
\DeclareMathOperator{\diag}{diag}
\DeclareMathOperator{\Ker}{ker}
\newcommand{\R}{\mathbb R}
\newcommand{\one}{\mathbf 1}
\newcommand{\pp}{\mathbf p}
\newcommand{\EE}{\mathcal E}
\newcommand{\TT}{\mathcal T}
\newcommand{\MM}{\mathcal M}
\newcommand{\BB}{\mathcal B}
\newcommand{\CC}{\mathcal C}
\newcommand{\pos}[1]{\left(#1\right)_{+}}
\title{On the Brouwer-type Conjecture for Signless Laplacian Eigenvalues of Graphs}
\author{
Ruisong Yuan\textsuperscript{1,2}\thanks{Email: \href{mailto:doubendouyrs@sjtu.edu.cn}{doubendouyrs@sjtu.edu.cn}},\quad
Xiao-Dong Zhang\textsuperscript{1}\thanks{Corresponding author. Email: \href{mailto:xiaodong@sjtu.edu.cn}{xiaodong@sjtu.edu.cn}}
\\[0.4em]
{\small\textsuperscript{1}School of Mathematical Sciences, Shanghai Jiao Tong University, Shanghai, China}\\[0.4em]
{\small\textsuperscript{2}Shanghai Innovation Institute, Shanghai, China}
}
\date{}

\begin{document}
\maketitle
\begin{abstract}
Motivated by Brouwer's conjecture, Ashraf, Omidi and Tayfeh-Rezaie proposed the following Brouwer-type conjecture that  for every graph $G$ on
$n$ vertices with $m$ edges, the sum $S_k^+(G)$ of its $k$ largest signless Laplacian
eigenvalues satisfies $S_k^+(G)\le m+\binom{k+1}{2}$ for $k=1, \ldots,  n$. In this paper, we prove that the above conjecture holds. Moreover, the equality holds
if and only if $k=1$ and $G$ is either  star $K_{1,a}$ or triangle $K_3$ with adding some isolated vertices.
 For split graphs, properties of block signless Laplacian matrices based on clique and independent set are adapted. While for non-split graphs, some spectral graph substructure are used to control the sum of signless Laplacian eigenvalues.
\end{abstract}


\section{Introduction}\label{sec:intro}
All graphs in this paper are finite, simple, and undirected.
Let $G$ be a graph with vertex set $V(G)$ and edge set $E(G)$, and put
$n=|V(G)|$ and $m=|E(G)|$. Let $A(G)$ be its adjacency matrix and
$D(G)$ the diagonal matrix of vertex degrees. We denote the signless
Laplacian matrix of $G$ by $Q(G)=D(G)+A(G)$ and write its eigenvalues
in non-increasing order as $q_1(G)\ge\cdots\ge q_n(G)\ge0$. For
$1\le k\le n$, set
\[
 S_k^+(G)=\sum_{i=1}^k q_i(G).
\]

Brouwer's Laplacian conjecture concerns the corresponding sum of
eigenvalues of $L(G)=D(G)-A(G)$ and proposes the upper bound
$m+\binom{k+1}{2}$. \citet*{AOT} proposed the following conjecture.

\begin{conjecture}[{\citet{AOT}}]\label{conj:signless}
Let $G$ be a graph with $n$ vertices and $m$ edges. Then, for every
integer $1\le k\le n$,
\begin{equation}\label{eq:conjecture}
 S_k^+(G)\le m+\binom{k+1}{2}.
\end{equation}
\end{conjecture}

It is known that Conjecture~\ref{conj:signless} holds for several graph
classes and for $k=1,2,n-2,n-1,n$. These classes include regular graphs
(\citet{AOT}), unicyclic and bicyclic graphs (\citet{YY}), and threshold
graphs (\citet{HPTT}). The cases $k=1,n-1,n$ follow from standard
eigenvalue bounds. \citet*{ZHS}, with the corrections in \citet{DZSH},
proved the case $k=2$. The reduction of \citet{CHJL} then gives the
case $k=n-2$.

Related extremal problems for the gap $m+3-S_2^+(G)$ were studied by
\citet{OLRC}. The minimizing graphs were determined for fixed order
$n\ge9$ by \citet{ZHS}, with the corrigendum of \citet{DZSH}, and for
a fixed number of edges $m\ge4$ by \citet*{ZDH}. \citet{ADKMO}
obtained sharp upper and lower bounds for $S_k^+(G)$ using eigenvalue
interlacing.

In this paper, the main results can be stated as follows:

\begin{theorem}\label{thm:main}
Let $G$ be a simple graph on $n\ge1$ vertices and $m$ edges. Then
\[
 S_k^+(G)\le m+\binom{k+1}{2}, \qquad k=1,\ldots,n.
\]
Moreover, equality holds if and only if $k=1$ and
\[
 G\cong K_{1,a}\sqcup tK_1\quad(a\ge1,\ t\ge0),
 \qquad\text{or}\qquad
 G\cong K_3\sqcup tK_1\quad(t\ge0).
\]
\end{theorem}

The rest of the paper is organized as follows. Section~\ref{sec:identity}
introduces the preliminaries. Sections~\ref{sec:split} and~\ref{sec:obstructions}
prove Conjecture~\ref{conj:signless} for split and non-split graphs,
respectively. Section~\ref{sec:conclusion} contains conclusions and remarks.

\section{Preliminaries}\label{sec:identity}

We write $I_r$ and $J_r$ for the $r\times r$ identity and all-ones
matrices, respectively, and $\one$ for the all-ones column vector.
For $t\in\R$, put $t_+=\max\{t,0\}$. We denote the Frobenius norm
by $\|\cdot\|_F$. All pairwise sums are over unordered pairs of
distinct vertices.

\subsection{The variational reduction}

We first recall the maximum principle of \citet{Fan}.
The projection reduction is also used by \citet[Lemma~5.1]{KT}.

\begin{lemma}\label{lem:variational}
If $Q$ is a real symmetric $n\times n$ matrix with eigenvalues
$q_1\ge\cdots\ge q_n$, then, for every integer $1\le k\le n$,
\[
 \sum_{i=1}^kq_i
 =\max\{\tr(PQ):P=P^{\mathsf T}=P^2,\ \tr P=k\}.
\]
\end{lemma}
\begin{proof}
Choose an orthonormal eigenbasis $u_1,\ldots,u_n$ for $Q$. For any
projection in the indicated set, put $\alpha_i=u_i^{\mathsf T}Pu_i$.
Then $0\le\alpha_i\le1$ and $\sum_i\alpha_i=k$. Moreover,
$\tr(PQ)=\sum_iq_i\alpha_i$. Since
$\sum_{i\le k}(1-\alpha_i)=\sum_{i>k}\alpha_i$, we have
\begin{align*}
 \sum_{i=1}^kq_i-\sum_{i=1}^nq_i\alpha_i
 &=\sum_{i\le k}q_i(1-\alpha_i)-\sum_{i>k}q_i\alpha_i\\
 &\ge q_k\left(\sum_{i\le k}(1-\alpha_i)-\sum_{i>k}\alpha_i\right)=0.
\end{align*}
Equality holds for the projection onto
$\operatorname{span}\{u_1,\ldots,u_k\}$.
\end{proof}

Label the vertices of $G$ by $1,\ldots,n$, and fix a rank-$k$
projection $P$ for the rest of this section. Set
\begin{equation}\label{eq:entries}
 p_i=P_{ii},\quad a_i=1-p_i,\quad c_{ij}=P_{ij}\ (i\ne j),\quad
 \pp=(p_1,\ldots,p_n)^{\mathsf T}.
\end{equation}
The identities $P^2=P$ and $\tr P=k$ give
\begin{equation}\label{eq:row}
 0\le p_i\le1,\qquad \sum_i p_i=k,\qquad
 \sum_{j\ne i}c_{ij}^2=p_i-p_i^2=p_i a_i.
\end{equation}
In particular, $p_i\in\{0,1\}$ implies $c_{ij}=0$ for all $j\ne i$.
Also $P$ is positive semidefinite, so its $2\times2$ principal minors give
\begin{equation}\label{eq:cauchy}
 |c_{ij}|\le\sqrt{p_ip_j}.
\end{equation}

Each edge contributes the rank-one matrix
$(e_i+e_j)(e_i+e_j)^{\mathsf T}$ to $Q(G)$. Consequently,
\[
 \tr(PQ(G))=\sum_{ij\in E(G)}(p_i+p_j+2c_{ij}).
\]
Define
\begin{equation}\label{eq:f}
 f_{ij}=p_i+p_j+2c_{ij}-1,\qquad
 \EE_G(P)=\tr(PQ(G))-m-\binom{k+1}{2}.
\end{equation}
Then
\[
 \EE_G(P)=\sum_{ij\in E(G)}f_{ij}-\binom{k+1}{2}.
\]
By Lemma~\ref{lem:variational}, it suffices to prove $\EE_G(P)\le0$
for every such $P$.

\subsection{An identity for the excess}
In this section, we derive an identity expressing the excess $\EE_G(P)$
in~\eqref{eq:f} as $1/2$ minus a sum of nonnegative terms.
In Section~\ref{sec:obstructions}, we show that this sum exceeds
$1/2$ for non-split graphs, which gives $\EE_G(P)<0$.
We begin by defining
\begin{equation}\label{eq:pairloss}
 \MM_{ij}=
 \begin{cases}
 (a_i^{-1}+a_j^{-1}-1)(c_{ij}-a_i a_j)^2,&a_i a_j>0,\\
 0,&a_i a_j=0,
 \end{cases}
 \qquad
 \BB_{ij}=\MM_{ij}+f_{ij}.
\end{equation}

\begin{lemma}\label{lem:pair}
Both $\MM_{ij}$ and $\BB_{ij}$ are nonnegative. If $a_i a_j>0$, then
\begin{align}
 \BB_{ij}
 &=\frac{1-p_ip_j}{a_i a_j}c_{ij}^2
       +2p_ip_jc_{ij}
       +p_ip_j(p_i+p_j-p_ip_j),\label{eq:Bexpanded}\\
 \BB_{ij}
 &=\frac{1-p_ip_j}{a_i a_j}
   \left(c_{ij}+\frac{p_ip_j a_i a_j}{1-p_ip_j}\right)^2
   +\frac{p_ip_j(p_i+p_j-2p_ip_j)}{1-p_ip_j}.
   \label{eq:Bcompleted}
\end{align}
If $p_i=1$, then $\MM_{ij}=0$ and $\BB_{ij}=p_j$.
\end{lemma}
\begin{proof}
For positive $a_i,a_j\le1$, the coefficient in $\MM_{ij}$ is
nonnegative. To expand $\BB_{ij}$, note that
\[
 a_i+a_j-a_i a_j=1-p_ip_j.
\]
The linear term in $c_{ij}$ after adding $f_{ij}$ is therefore
$[-2(1-p_ip_j)+2]c_{ij}=2p_ip_jc_{ij}$. The constant term is
\[
 a_i a_j(1-p_ip_j)+p_i+p_j-1
   =p_ip_j(p_i+p_j-p_ip_j).
\]
This proves~\eqref{eq:Bexpanded}. Completing its square gives
\eqref{eq:Bcompleted}. The remaining numerator is nonnegative because
$p_i+p_j-2p_ip_j=p_i(1-p_j)+p_j(1-p_i)\ge0$.

If $a_i=0$, then $p_i=1$ and $c_{ij}=0$ by~\eqref{eq:row}. In this
case $\MM_{ij}=0$ and $\BB_{ij}=p_j\ge0$. The case $a_j=0$ is the
same. This also covers the case $p_i=p_j=1$.
\end{proof}

Set
\begin{equation}\label{eq:globaldefs}
 \tau=\sum_i p_i a_i,\qquad r_0=|\{i:p_i=1\}|,\qquad
 R=\pp^{\mathsf T}(I-P)\pp=\|(I-P)\pp\|^2,
\end{equation}
and define the vertex loss
\begin{equation}\label{eq:v}
 v(p)=\frac12p(1-p)^2(2-p),\qquad 0\le p\le1.
\end{equation}
Finally, define
\begin{equation}\label{eq:T}
 \TT_G(P)=\sum_{ij\in E(G)}\MM_{ij}
       +\sum_{ij\notin E(G)}\BB_{ij}.
\end{equation}
We call a diagonal entry $p_i$ \emph{saturated} if $p_i=1$;
thus $r_0$ counts the saturated diagonal entries.
All four quantities $\tau,r_0,R,v(p)$ are nonnegative. The equality of
the two expressions for $R$ follows from $(I-P)^2=I-P$.
By~\eqref{eq:row}, $\tau=2\sum_{i<j}c_{ij}^2$: it is twice the sum of
the squared entries strictly above the diagonal of $P$.
The term $R$ is the squared distance
of $\pp$ from the range of $P$.
The following identity plays a central role in our proof.

\begin{lemma}\label{lem:identity}
For every graph $G$ and every rank-$k$ real orthogonal projection $P$,
\begin{align}
 \EE_G(P)&=\tau-\frac{\tau^2}{2}-r_0-R
                    -\sum_i v(p_i)-\TT_G(P)\label{eq:identity}\\
 &=\frac12-\frac{(\tau-1)^2}{2}-r_0-R
                    -\sum_i v(p_i)-\TT_G(P).\label{eq:roadmap}
\end{align}
\end{lemma}
\begin{proof}
Since $\BB_{ij}=\MM_{ij}+f_{ij}$, the definition of $\TT_G(P)$ gives
\begin{equation}\label{eq:cancelT}
 \sum_{i<j}\BB_{ij}-\TT_G(P)=\sum_{ij\in E(G)}f_{ij}.
\end{equation}
It remains to compute the sum on the left.

For $r=2,3,4$, write $s_r=\sum_i p_i^r$. In particular,
$s_2=k-\tau$. The three parts of~\eqref{eq:Bexpanded} can be summed
separately. For its quadratic part, omit pairs with $a_i a_j=0$ and use
\eqref{eq:row}:
\begin{align}
 \sum_{\substack{i<j\\ a_i a_j>0}}
 (a_i^{-1}+a_j^{-1}-1)c_{ij}^2
 &=\sum_{i:a_i>0}\frac1{a_i}\sum_{j\ne i}c_{ij}^2
          -\sum_{i<j}c_{ij}^2\notag\\
 &=\sum_{i:a_i>0}p_i-\frac\tau2
   =k-r_0-\frac\tau2.\label{eq:sum1}
\end{align}
In the first line, the omitted row terms are zero because
$p_i=1$ implies $c_{ij}=0$ for every $j\ne i$.

For the linear part, use the definition of $R$:
\begin{align}
 2\sum_{i<j}p_ip_jc_{ij}
 &=\pp^{\mathsf T}P\pp-\sum_i p_i^3\notag\\
 &=\pp^{\mathsf T}\pp-R-s_3
   =s_2-s_3-R.\label{eq:sum2}
\end{align}
For the constant part, the two elementary product expansions are
\[
 \sum_{i<j}p_ip_j(p_i+p_j)=ks_2-s_3,
 \qquad
 \sum_{i<j}p_i^2p_j^2=\frac{s_2^2-s_4}{2}.
\]
Hence
\begin{equation}\label{eq:sum3}
 \sum_{i<j}p_ip_j(p_i+p_j-p_ip_j)
 =ks_2-s_3-\frac{s_2^2-s_4}{2}.
\end{equation}
If $p_i=1$, the constant term equals $p_j=\BB_{ij}$, while the pair
contributes zero to the quadratic and linear sums.
The case $p_j=1$ is symmetric.

Adding~\eqref{eq:sum1}--\eqref{eq:sum3} yields
\begin{equation}\label{eq:sumBraw}
 \sum_{i<j}\BB_{ij}
 =k-r_0-R-\frac\tau2+(k+1)s_2-2s_3-\frac{s_2^2}{2}+\frac{s_4}{2}.
\end{equation}
On the other hand, expansion of~\eqref{eq:v} gives
\[
 V:=\sum_i v(p_i)=k-\frac52s_2+2s_3-\frac12s_4.
\]
Substitute $-2s_3+s_4/2=k-5s_2/2-V$ into~\eqref{eq:sumBraw}, and
then substitute $s_2=k-\tau$. This gives
\begin{align*}
 \sum_{i<j}\BB_{ij}
 &=2k-r_0-R-\frac\tau2+\left(k-\frac32\right)s_2
        -\frac{s_2^2}{2}-V\\
 &=\binom{k+1}{2}+\tau-\frac{\tau^2}{2}-r_0-R-V.
\end{align*}
Finally, combining this identity with~\eqref{eq:cancelT}
and~\eqref{eq:f} proves the lemma.
\end{proof}

The following is a direct corollary of Lemma~\ref{lem:identity}.

\begin{corollary}\label{cor:half}
For every graph $G$ and every $1\le k\le |V(G)|$,
\[
 S_k^+(G)\le |E(G)|+\binom{k+1}{2}+\frac12.
\]
\end{corollary}
\begin{proof}
Discard the nonnegative terms $\sum_i v(p_i)$ and $\TT_G(P)$ in
\eqref{eq:identity}, and use $\tau-\tau^2/2=1/2-(\tau-1)^2/2$.
Then apply Lemma~\ref{lem:variational}.
\end{proof}

\begin{remark}\label{ex:coordinate}
Let $P$ project onto the coordinates in a set $U$ of size $k$.
Then $\tau=R=\sum_i v(p_i)=0$ and $r_0=k$.
A direct inspection of~\eqref{eq:pairloss} gives
\[
 \TT_G(P)=|E(G[V(G)\setminus U])|
           +\#\{\text{nonedges within }U\}.
\]
Thus~\eqref{eq:identity} reads
\[
 \EE_G(P)=-k-\TT_G(P).
\]
Direct edge counting gives the same result:
$\EE_G(P)=|E(G[U])|-|E(G[V(G)\setminus U])|-\binom{k+1}{2}$.
The term $-r_0$ accounts for the $k$ saturated coordinates.
\end{remark}

\subsection{Lower bounds for the pairwise terms}

The bounds $\ell_E$ and $\ell_N$ below will be used for
$2K_2$ and $C_5$, while the stronger bounds $E$ and $N$ will be
used for $C_4$. For $p,q\in[0,1]$, define
\begin{equation}\label{eq:scalar}
 \ell_E(p,q)=\pos{(1-p)(1-q)-\sqrt{pq}}^2,\qquad
 \ell_N(p,q)=\frac{pq}{4}\bigl(3(p+q)-p^2-q^2\bigr).
\end{equation}
We will also use the stronger bounds
\begin{equation}\label{eq:strong-scalar}
 E(p,q)=\frac{1-pq}{(1-p)(1-q)}\ell_E(p,q),\qquad
 N(p,q)=\frac{pq(p+q-2pq)}{1-pq},
\end{equation}
where $E(p,q)=0$ if $(1-p)(1-q)=0$, and $N(1,1)=1$.

The following lemma gives lower bounds for $\MM_{ij}$ and $\BB_{ij}$
in terms of the diagonal entries $p_i$ and $p_j$ alone. This eliminates
$c_{ij}$ from the local estimates used in Section~\ref{sec:obstructions}.

\begin{lemma}\label{lem:scalar}
For every pair of distinct vertices,
\[
 \MM_{ij}\ge E(p_i,p_j)\ge\ell_E(p_i,p_j),\qquad
 \BB_{ij}\ge N(p_i,p_j)\ge\ell_N(p_i,p_j).
\]
All four scalar functions are continuous and nonnegative on $[0,1]^2$.
\end{lemma}
\begin{proof}
Write $p=p_i$, $q=p_j$, $D=(1-p)(1-q)$, and $c=c_{ij}$.
For $D>0$, the definition of $\MM_{ij}$ and~\eqref{eq:cauchy} give
\[
 \MM_{ij}=\frac{1-pq}{D}(c-D)^2
 \ge\frac{1-pq}{D}\pos{D-\sqrt{pq}}^2=E(p,q).
\]
Moreover, $(1-pq)/D=(1-p)^{-1}+(1-q)^{-1}-1\ge1$.
The bound $\BB_{ij}\ge N(p,q)$ follows from~\eqref{eq:Bcompleted}.

To compare $N$ and $\ell_N$, first suppose $0<u:=pq<1$ and put
$s=p+q$, $t=\sqrt{u}$. Direct subtraction gives
\begin{align*}
 \frac{4(1-u)}{u}\bigl(N(p,q)-\ell_N(p,q)\bigr)
 &=(1-u)s^2+(1+3u)s-10u+2u^2\\
 &=(s-2t)\bigl((1-t^2)(s+2t)+1+3t^2\bigr)
   +2t(1-t)^3\ge0.
\end{align*}
If $pq=0$, both functions vanish. If $p=1$, then
$\MM_{ij}=E(1,q)=0$ and $\BB_{ij}=N(1,q)=q$, whereas
\[
 \ell_N(1,q)=\frac q4(2+3q-q^2)\le q.
\]
This also treats $p=q=1$; the case $q=1$ is symmetric.

Only the endpoint continuity of $E$ and $N$ needs explanation.
For $D>0$, $0\le E(p,q)\le(1-pq)D$, so $E\to0$ as $D\to0$.
Writing $a=1-p$ and $b=1-q$, we have
\[
 \frac{p+q-2pq}{1-pq}=1-\frac{ab}{a+b-ab}.
\]
Since $0\le ab/(a+b-ab)\le\min(a,b)$, the prescribed value
$N(1,1)=1$ is the limit. The other assertions follow from the formulas.
\end{proof}

\subsection{Characterization of split graphs}

The following characterization by forbidden induced subgraphs
will be used in our proof for non-split graphs.

\begin{lemma}[{\citet{FH}}]\label{lem:split-characterization}
A graph is split if and only if it has no induced $2K_2$, $C_4$, or $C_5$.
\end{lemma}
\begin{proof}
Every induced subgraph of a split graph is split, whereas none of the
three listed graphs is split. Conversely, suppose that $G$ contains none
of them. Choose a maximum clique $C$ for which $\sum_{v\in C}\deg(v)$
is as large as possible. Suppose that $x,y\notin C$ are adjacent, and put
$A=C\setminus N(x)$ and $B=C\setminus N(y)$. Both sets are nonempty.
If they are incomparable, vertices $a\in A\setminus B$ and
$b\in B\setminus A$, together with $x,y$, induce the cycle
$a-b-x-y-a$. Thus, after interchanging $x$ and $y$, we may assume
$A\subseteq B$.

If $A$ contains two vertices, they and $x,y$ induce $2K_2$.
Hence $A=\{a\}$. Replacing $a$ by $x$ gives another maximum clique,
so $\deg(a)\ge\deg(x)$. Since $y$ is adjacent to $x$ but not to $a$,
there is a vertex $z\notin C$, different from $x,y$, that is adjacent
to $a$ and not to $x$. We must have $yz\in E(G)$, for otherwise $az$
and $xy$ induce $2K_2$. There is a vertex $b\in B\setminus\{a\}$,
since otherwise $(C\setminus\{a\})\cup\{x,y\}$ would be a larger clique.
If $bz$ is an edge, then $b,x,y,z$ induce $C_4$; if it is a nonedge,
then $a,b,x,y,z$ induce the cycle $a-b-x-y-z-a$.
Both are impossible. Thus $V(G)\setminus C$ is independent.
\end{proof}

\section{Conjecture for split graphs}\label{sec:split}

In this section, we prove that the conjecture holds for split graphs and determine when equality holds.

Let $G$ be a split graph with a fixed partition $V(G)=C\sqcup I$,
where $C$ is a clique of size $c$ and $I$ is an independent set of
size $b$. Let $B$ be the $c\times b$ zero--one matrix of edges between
$C$ and $I$, and write
\[
 r_i=\sum_j B_{ij},\qquad s_j=\sum_i B_{ij}.
\]
With respect to the partition $V(G)=C\sqcup I$, the signless
Laplacian matrix has the block form
\begin{equation}\label{eq:Qsplit}
 Q(G)=\begin{pmatrix}A&B\\B^{\mathsf T}&D\end{pmatrix},\qquad
 A=\diag(r_1,\ldots,r_c)+(c-2)I_c+J_c,\quad
 D=\diag(s_1,\ldots,s_b),
\end{equation}
where $J_c=\one\one^{\mathsf T}$. In particular,
\begin{equation}\label{eq:splitcounts}
 m=\sum_i r_i+\binom c2,\qquad
 \tr A=\sum_i r_i+c(c-1),\qquad \tr A-m=\binom c2.
\end{equation}

\subsection{A block-matrix identity}
Write a rank-$k$ orthogonal projection in the same blocks as
\begin{equation}\label{eq:Psplit}
 P=\begin{pmatrix}X&Y\\Y^{\mathsf T}&W\end{pmatrix},\qquad H=I_c-X.
\end{equation}
The projection relations imply
\begin{equation}\label{eq:Hrel}
 H^2+YY^{\mathsf T}=H,\qquad
 \tr W=k-c+\tr H.
\end{equation}

We will choose $\mu$ to bound the degrees in $I$, so that
$\tr(W(D-\mu I_b))\le0$. For the remaining terms, we use
$H^2+YY^{\mathsf T}=H$ to complete two squares involving the same matrix
$M$. The shift by $2I_c$ supplies the term $-2\tr H$ in the first
square. For now, choose any real number $\mu$ such that
\[
 M=A-\mu I_c+2I_c\succ0,
\]
and put
\begin{equation}\label{eq:LHLY}
\begin{aligned}
 \kappa&=\tr\bigl((BB^{\mathsf T}+I_c)M^{-1}\bigr),\\
 L_H&=\|M^{1/2}H-M^{-1/2}\|_F^2,\qquad
 L_Y=\|M^{1/2}Y-M^{-1/2}B\|_F^2.
\end{aligned}
\end{equation}

\begin{lemma}\label{lem:block}
With the above notation,
\begin{equation}\label{eq:blockidentity}
 \tr(PQ(G))=\tr A+(k-c)\mu+\kappa-L_H-L_Y
                      +\tr\bigl(W(D-\mu I_b)\bigr).
\end{equation}
The last term is nonpositive if $D\preceq\mu I_b$.
\end{lemma}
\begin{proof}
Since $X=I_c-H$, block multiplication gives
\[
 \tr(PQ(G))=\tr A-\tr(AH)+2\tr(B^{\mathsf T}Y)+\tr(WD).
\]
Substitute $\tr(WD)=\mu\tr W+\tr(W(D-\mu I_b))$ and
use~\eqref{eq:Hrel}. This gives
\[
 \tr(PQ(G))=\tr A+(k-c)\mu-\Phi
                   +\tr\bigl(W(D-\mu I_b)\bigr),
\]
where
\[
 \Phi=\tr\bigl((A-\mu I_c)H\bigr)-2\tr(B^{\mathsf T}Y).
\]

It remains to write $\Phi=L_H+L_Y-\kappa$. Since
$A-\mu I_c=M-2I_c$, relation~\eqref{eq:Hrel} yields
\begin{align*}
 \Phi
 &=\tr(MH)-2\tr H-2\tr(B^{\mathsf T}Y)\\
 &=\bigl[\tr(MH^2)-2\tr H\bigr]
   +\bigl[\tr(MYY^{\mathsf T})-2\tr(B^{\mathsf T}Y)\bigr].
\end{align*}
Expanding the squares gives
\begin{align*}
 L_H&=\tr(MH^2)-2\tr H+\tr M^{-1},\\
 L_Y&=\tr(MYY^{\mathsf T})-2\tr(B^{\mathsf T}Y)
                         +\tr(B^{\mathsf T}M^{-1}B).
\end{align*}
These equalities use cyclicity of the trace; no commutation of $H$ and
$M$ is assumed. Their constant terms add up to $\kappa$, so
\[
 L_H+L_Y=\Phi+\kappa
 =\tr\bigl((A-\mu I_c)H\bigr)-2\tr(B^{\mathsf T}Y)+\kappa.
\]
Substitution in the expression for $\tr(PQ(G))$
proves~\eqref{eq:blockidentity}. Finally, $W$ is positive semidefinite,
so if $s_j\le\mu$ for all $j$, then
\[
 \tr\bigl(W(D-\mu I_b)\bigr)=\sum_{j=1}^b W_{jj}(s_j-\mu)\le0.
\]
\end{proof}

\begin{remark}
Lemma~\ref{lem:identity} and Lemma~\ref{lem:block} give two
decompositions of the same excess. For the fixed partition and $\mu$, put
\[
 \beta_\mu=\binom c2+(k-c)\mu+\kappa-\binom{k+1}{2}.
\]
Since $\tr A-m=\binom c2$, comparison gives
\[
 \begin{aligned}
 r_0+R+\sum_i v(p_i)+\TT_G(P)
 &=\tau-\frac{\tau^2}{2}-\beta_\mu\\
 &\quad+L_H+L_Y+\tr\bigl(W(\mu I_b-D)\bigr).
 \end{aligned}
\]
Thus the pairwise losses and the matrix squares need not agree
individually. The first decomposition is suited to induced subgraphs;
the second retains the clique--independent-set partition and the rank
restriction.
\end{remark}

\subsection{Auxiliary matrix estimates}

\begin{lemma}\label{lem:rankloss}
Suppose $M\succeq I_c$ and let $\theta=\lambda_{\min}(M)$. Define
$\phi(t)=t-2+t^{-1}$. Then
\begin{equation}\label{eq:rankloss}
 L_H\ge(c-k)_+\phi(\theta).
\end{equation}
\end{lemma}
\begin{proof}
Since $X$ is a principal block of a rank-$k$ projection,
$\rank X\le k$. Therefore $d:=\dim\Ker X\ge(c-k)_+$.
Choose orthonormal vectors $z_1,\ldots,z_d$ spanning $\Ker X$ and
extend them to an orthonormal basis of $\R^c$. On this kernel,
$Hz_j=z_j$. Keeping just the corresponding terms in the Frobenius norm,
\begin{align*}
 L_H&\ge\sum_{j=1}^d
      \|(M^{1/2}H-M^{-1/2})z_j\|^2\\
 &=\sum_{j=1}^d z_j^{\mathsf T}(M-2I_c+M^{-1})z_j.
\end{align*}
The eigenvalues of $M-2I_c+M^{-1}$ are $\phi(\lambda)$, where
$\lambda$ ranges over the eigenvalues of $M$. Since
$\phi'(t)=1-t^{-2}\ge0$ for $t\ge1$, each term is at least
$\phi(\theta)$. This proves~\eqref{eq:rankloss}.
\end{proof}

For $c\ge1$ and positive numbers $d_1,\ldots,d_c$, put
\[
 u_i=d_i^{-1},\qquad U=\diag(u_1,\ldots,u_c),\qquad
 u=(u_1,\ldots,u_c)^{\mathsf T},\qquad \sigma=\sum_i u_i.
\]
Direct multiplication gives
\begin{equation}\label{eq:inverse}
 (\diag(d_1,\ldots,d_c)+J_c)^{-1}
       =U-\frac{uu^{\mathsf T}}{1+\sigma}.
\end{equation}
Indeed, multiplication of $\diag(d_i)+J_c$ by $U$ gives
$I_c+\one u^{\mathsf T}$, while multiplication by $u$ gives
$(1+\sigma)\one$.

\begin{lemma}\label{lem:scalarM}
Suppose $d_i\ge2$ for every $i$, and let
$M=\diag(d_i)+J_c$ and $\theta=\lambda_{\min}(M)$. Then
\begin{equation}\label{eq:theta}
 \theta\ge2,\qquad \phi(\theta)\ge\frac1{1+\sigma}.
\end{equation}
\end{lemma}
\begin{proof}
The first inequality follows from $M\succeq2I_c$. By~\eqref{eq:inverse},
the sum of the absolute values in row $i$ of $M^{-1}$ is
\[
 u_i-\frac{u_i^2}{1+\sigma}
       +\frac{u_i(\sigma-u_i)}{1+\sigma}
 =\frac{u_i(1+2\sigma-2u_i)}{1+\sigma}
 \le\frac{\sigma}{1+\sigma}.
\]
The last inequality follows from $(\sigma-u_i)(1-2u_i)\ge0$.
Bounding the spectral radius by the maximum absolute row sum gives
\[
 \lambda_{\max}(M^{-1})\le\frac{\sigma}{1+\sigma},\qquad
 \theta\ge\frac{1+\sigma}{\sigma}.
\]
If $\sigma\le1$, monotonicity of $\phi$ on $[1,\infty)$ gives
\[
 \phi(\theta)\ge\phi(1+\sigma^{-1})
    =\frac1{\sigma(1+\sigma)}\ge\frac1{1+\sigma}.
\]
If $\sigma\ge1$, the bound $\theta\ge2$ instead gives
$\phi(\theta)\ge\phi(2)=1/2\ge1/(1+\sigma)$.
\end{proof}

\subsection{The split bound}

\begin{theorem}\label{thm:split}
For every split graph $G$ and every rank-$k$ orthogonal projection $P$,
$\EE_G(P)\le0$.
\end{theorem}
\begin{proof}
The assertion is immediate for an edgeless graph. We may therefore
assume that the clique is nonempty. We distinguish whether an
independent vertex is adjacent to the whole clique. The corresponding
degree bounds are $c$ and $c-1$, which determine the choices of $\mu$
in Lemma~\ref{lem:block}.

\medskip\noindent\textbf{Case 1.} Some vertex of $I$ is adjacent to every
vertex of $C$. Then $r_i\ge1$ for all $i$. Choose $\mu=c$, so that
\[
 D\preceq cI_b,\qquad M=\diag(r_i)+J_c\succeq I_c.
\]
We claim that $\kappa\le c$. The matrix $M^{-1}$ has nonpositive
off-diagonal entries by~\eqref{eq:inverse}. On the other hand,
$BB^{\mathsf T}+I_c-M$ has zero diagonal and nonnegative off-diagonal
entries: any two clique vertices share the independent neighbor assumed
in this case. Consequently,
\[
 \kappa-c=\tr\bigl((BB^{\mathsf T}+I_c-M)M^{-1}\bigr)\le0.
\]
Discard $L_H,L_Y$ and the last term in~\eqref{eq:blockidentity}.
Using~\eqref{eq:splitcounts},
\begin{align}
 \EE_G(P)
 &\le\binom c2+(k-c)c+c-\binom{k+1}{2}\notag\\
 &=-\frac{(k-c)(k-c+1)}2\le0.\label{eq:universalclose}
\end{align}
The last inequality uses the integrality of $k-c$.

\medskip\noindent\textbf{Case 2.} No vertex of $I$ is adjacent to every
vertex of $C$. Then $s_j\le c-1$ for all $j$. We may assume every $r_i\ge1$. Indeed, if
some clique vertex $v$ has $r_v=0$, move $v$ from the clique to the
independent set. The latter remains independent, while $v$ is adjacent
to all vertices of the new clique. If the new clique is nonempty, Case~1
applies to this new partition. If it is empty, the graph is edgeless,
which has already been treated.

Choose $\mu=c-1$. Now
\[
 M=\diag(r_i+1)+J_c\succeq2I_c,\qquad D\preceq(c-1)I_b.
\]
Set $u_i=(r_i+1)^{-1}$ and $\sigma=\sum_i u_i$.
The off-diagonal entries of $BB^{\mathsf T}$ are nonnegative, whereas
those of $M^{-1}$ are nonpositive. Thus
\begin{align}
 \kappa
 &\le\sum_i(r_i+1)\left(u_i-\frac{u_i^2}{1+\sigma}\right)\notag\\
 &=c-\frac{\sigma}{1+\sigma}.\label{eq:kappaproper}
\end{align}
For the integer $\ell=c-k$, Lemma~\ref{lem:block} gives
\begin{align}
 \EE_G(P)
 &\le\binom c2+(k-c)(c-1)+\kappa-L_H-L_Y-\binom{k+1}{2}\notag\\
 &=\kappa-c+\frac{\ell(3-\ell)}2-L_H-L_Y.\label{eq:properclose}
\end{align}
If $\ell\le0$ or $\ell\ge3$, the quadratic term is nonpositive,
so~\eqref{eq:kappaproper} proves the result.
If $\ell=2$, Lemma~\ref{lem:rankloss} and $\theta\ge2$ give
$L_H\ge2\phi(2)=1$, exactly cancelling the quadratic term.
Finally, if $\ell=1$, Lemmas~\ref{lem:rankloss} and~\ref{lem:scalarM}
give $L_H\ge1/(1+\sigma)$. Therefore
\[
 \EE_G(P)\le-\frac{\sigma}{1+\sigma}+1
                  -\frac1{1+\sigma}-L_Y=-L_Y\le0.
\]
This covers all integer values of $k$.
\end{proof}

\subsection{Equality cases}\label{subsec:split-equality}
We now determine the equality cases by examining the inequalities
used in the proof of Theorem~\ref{thm:split}.

\begin{theorem}\label{thm:split-equality}
Let $G$ be a split graph with $n$ vertices and $m$ edges, and let
$k\in\{1,\ldots,n\}$. Then
\[
 S_k^+(G)=m+\binom{k+1}{2}
\]
if and only if $k=1$ and deleting all isolated vertices from $G$
leaves a star $K_{1,a}$, $a\ge1$, or a triangle $K_3$.
In particular, the split bound is strict at every rank $k\ge2$.
\end{theorem}

\begin{proof}
By Lemma~\ref{lem:variational}, the spectral equality holds precisely
when there is a rank-$k$ orthogonal projection $P$ with
$\EE_G(P)=0$. An edgeless graph has
$\EE_G(P)=-\binom{k+1}{2}<0$, so we assume that $G$ has an edge.
We use the block notation of~\eqref{eq:Qsplit}--\eqref{eq:LHLY}.

For $D\preceq\mu I_b$, write
\begin{equation}\label{eq:degree-loss}
 Z_\mu:=\tr\bigl(W(\mu I_b-D)\bigr)
       =\sum_j(\mu-s_j)W_{jj}\ge0.
\end{equation}
Then Lemma~\ref{lem:block} gives
\begin{equation}\label{eq:split-exact-excess}
 \EE_G(P)=\binom c2+(k-c)\mu+\kappa-\binom{k+1}{2}
             -L_H-L_Y-Z_\mu.
\end{equation}
The square terms satisfy
\begin{equation}\label{eq:equality-zero-squares}
 L_H=0\ \Longleftrightarrow\ H=M^{-1},\qquad
 L_Y=0\ \Longleftrightarrow\ Y=M^{-1}B.
\end{equation}
If both vanish, substituting in $H^2+YY^{\mathsf T}=H$ gives
\begin{equation}\label{eq:equality-compatibility}
 I_c+BB^{\mathsf T}=M.
\end{equation}

Also, a zero diagonal entry of $P$ forces its entire row and column to
vanish, since $P_{jj}=\sum_h P_{hj}^2$. Thus $Z_\mu=0$ forces $P$
to vanish on every independent vertex with $s_j<\mu$.
Finally, $X^2+YY^{\mathsf T}=X$ implies
$\Ker X\subseteq\Ker Y^{\mathsf T}$, and hence
\begin{equation}\label{eq:equality-rank-compatibility}
 \rank Y\le\rank X\le\rank P=k.
\end{equation}
In Case~1 of Theorem~\ref{thm:split}, $\mu=c$ and
\eqref{eq:split-exact-excess} reads
\[
 \EE_G(P)=-\frac{(k-c)(k-c+1)}2
              -(c-\kappa)-L_H-L_Y-Z_c.
\]
Every subtracted term is nonnegative. Equality therefore implies
\begin{equation}\label{eq:equality-universal-conditions}
 k\in\{c-1,c\},\qquad \kappa=c,\qquad L_H=L_Y=Z_c=0.
\end{equation}

If $s_j<c$, the degree loss gives $Y_{\cdot j}=0$, whereas $L_Y=0$
gives $Y=M^{-1}B$. Thus $B_{\cdot j}=0$. Every column of $B$ is
consequently either zero or $\one$. Let $a\ge1$ be the number of
nonzero columns. After deleting isolated vertices, $G$ is therefore
the join of a $c$-clique and an independent set of size $a$. Then
\[
 r_i=a\quad\text{for all }i,\qquad BB^{\mathsf T}=aJ_c,
 \qquad M=aI_c+J_c.
\]
The compatibility condition~\eqref{eq:equality-compatibility} becomes
\begin{equation}\label{eq:equality-universal-rigidity}
 I_c+aJ_c=aI_c+J_c,
 \qquad\text{equivalently}\qquad (a-1)(J_c-I_c)=0.
\end{equation}

If $c=1$, the graph is a nontrivial star with possible isolated vertices.
The condition $k\in\{c-1,c\}$ and $k\ge1$ gives $k=1$.

Suppose $c\ge2$. Equation~\eqref{eq:equality-universal-rigidity}
forces $a=1$. Let $u$ be the unique non-isolated independent vertex.
Then
\[
 M=I_c+J_c,\qquad M^{-1}=I_c-\frac{J_c}{c+1},\qquad
 X=\frac{J_c}{c+1},\qquad Y_{\cdot u}=\frac{\one}{c+1}.
\]
All rows and columns indexed by isolated vertices vanish. Writing
$w=P_{uu}$, the $u$-column of $XY+YW=Y$ gives
\[
 \frac{c}{c+1}Y_{\cdot u}+wY_{\cdot u}=Y_{\cdot u},
 \qquad\text{so}\qquad w=\frac1{c+1}.
\]
It follows that
\[
 P\big|_{C\cup\{u\}}=\frac{J_{c+1}}{c+1},
\]
a rank-one projection. Thus $k=1$, and
$1\in\{c-1,c\}$ with $c\ge2$ forces $c=2$.
The non-isolated part of $G$ is therefore a triangle.

In Case~2, a vertex with $r_i=0$ may be moved as in the preceding
proof, reducing to Case~1. We are left with $r_i\ge1$ and
$s_j\le c-1$. Choose $\mu=c-1$, and set
\[
 M=\diag(r_i+1)+J_c\succeq2I_c,\qquad
 u_i=(r_i+1)^{-1},\qquad \sigma=\sum_i u_i>0,\qquad \ell=c-k.
\]
Put
\[
 \eta:=\frac{2}{1+\sigma}
              \sum_{i<h}u_i u_h(BB^{\mathsf T})_{ih}\ge0.
\]
Keeping the off-diagonal terms in~\eqref{eq:inverse} gives
\[
 \kappa=c-\frac{\sigma}{1+\sigma}-\eta.
\]
Consequently,
\begin{equation}\label{eq:equality-proper-defects}
 \EE_G(P)=-\frac{\sigma}{1+\sigma}
            +\frac{\ell(3-\ell)}2-\eta-L_H-L_Y-Z_{c-1}.
\end{equation}

For $\ell\le0$ or $\ell\ge3$, this is strictly negative.
For $\ell=2$, the bound $L_H\ge1$ again makes it strictly negative.
When $\ell=1$, rewrite~\eqref{eq:equality-proper-defects} as
\[
 \EE_G(P)=-\eta-\left(L_H-\frac1{1+\sigma}\right)-L_Y-Z_{c-1}.
\]
All four terms subtracted on the right are nonnegative by
Lemmas~\ref{lem:rankloss} and~\ref{lem:scalarM}.
Equality would force $\eta=0$ and $L_Y=0$. Since every $u_i>0$,
the first condition gives
\[
 BB^{\mathsf T}=\diag(r_i)\succ0,\qquad \rank B=c.
\]
The second gives $Y=M^{-1}B$, and hence $\rank Y=c$. This contradicts
\eqref{eq:equality-rank-compatibility}, because $k=c-1$.
There is no equality in Case~2.

Conversely, the stated graphs attain equality. The signless
Laplacian spectrum of $K_{1,a}$ is $a+1,1^{[a-1]},0$: leaf vectors
of coordinate sum zero have eigenvalue one, and the remaining two
eigenvalues are those of $\left(\begin{smallmatrix}a&a\\1&1\end{smallmatrix}\right)$.
Here brackets indicate multiplicity. Thus $q_1=a+1=m+1$.
For $K_3$, the matrix $Q=I_3+J_3$ has eigenvalues $4,1,1$,
so again $q_1=4=m+1$.
Isolated vertices add only zero eigenvalues and no edges.
This proves sufficiency.
\end{proof}

\begin{remark}\label{rem:equality-projections}
For each equality graph in Theorem~\ref{thm:split-equality}, the rank-one
orthogonal projection maximizing $\tr(PQ(G))$ is unique.
For a star with its center listed first, let
$\mathbf v=(a,1,\ldots,1)^{\mathsf T}$, extended by zeros on isolated
vertices. The corresponding maximizing projections are
\[
 P=\frac{\mathbf v\mathbf v^{\mathsf T}}{a(a+1)}
 \quad\text{for }K_{1,a}\sqcup tK_1,
 \qquad
 P=\frac{J_3}{3}\oplus0
 \quad\text{for }K_3\sqcup tK_1.
\]
In both cases, the largest eigenvalue has multiplicity one. Every
maximizing rank-one projection is the orthogonal projection onto the
corresponding eigenspace, so the maximizing projection is unique.
\end{remark}

\section{Conjecture for non-split graphs}\label{sec:obstructions}

In this section, we prove that the inequality in
Conjecture~\ref{conj:signless} is strict for every non-split graph.

For a graph $H$ with vertex set $\{1,\ldots,d\}$, define
\begin{align}
 \CC_H(p_1,\ldots,p_d)
 &=\sum_i v(p_i)+\sum_{ij\in E(H)}\ell_E(p_i,p_j)
   +\sum_{ij\notin E(H)}\ell_N(p_i,p_j),\label{eq:CH}\\
 \widetilde\CC_H(p_1,\ldots,p_d)
 &=\sum_i v(p_i)+\sum_{ij\in E(H)}E(p_i,p_j)
   +\sum_{ij\notin E(H)}N(p_i,p_j).\label{eq:CH-strong}
\end{align}
where $(p_1,\ldots,p_d)\in[0,1]^d$. Here $v$ is defined
in~\eqref{eq:v}, and $\ell_E,\ell_N,E,N$ are defined
in~\eqref{eq:scalar}--\eqref{eq:strong-scalar}.

By~\eqref{eq:roadmap}, the excess $\EE_G(P)$ is $1/2$ minus a sum
of nonnegative terms. We obtain a lower bound for this sum by
retaining only the vertex and pair terms associated with an induced
subgraph and applying Lemma~\ref{lem:scalar}.
By Lemma~\ref{lem:split-characterization}, every non-split graph
contains an induced $2K_2$, $C_4$, or $C_5$.
We show that the corresponding local lower bounds, $\CC_{2K_2}$,
$\widetilde\CC_{C_4}$, and $\CC_{C_5}$, each exceed $1/2$.
It follows that $\EE_G(P)<0$ for
every rank-$k$ orthogonal projection $P$.

We establish the following local estimates.

\begin{theorem}\label{thm:local}
The following inequalities hold:
\[
 \CC_{2K_2}>\frac12,\qquad
 \CC_{C_5}\ge\frac7{12},\qquad
 \widetilde\CC_{C_4}\ge\frac{101}{200}.
\]
\end{theorem}

We prove the three estimates separately. The two-variable polynomial
identities used below are proved in the appendices, where all coefficients
are given explicitly.

\subsection{The \texorpdfstring{$2K_2$}{2K2} case}

The function below serves two purposes: it gives a product lower bound
for each nonedge loss, and bounds each edge loss together with its two
vertex terms using the sum of the endpoint values. The four nonedge
products then reduce the estimate to two variables. Put
\begin{equation}\label{eq:z0}
 z_0(p)=\frac{p^{3/2}(6-p)}5,\qquad 0\le p\le1.
\end{equation}

\begin{lemma}\label{lem:2k2-pair}
For $p,q\in[0,1]$,
\begin{align}
 \ell_N(p,q)&\ge z_0(p)z_0(q),\label{eq:2k2-nonedge}\\
 v(p)+v(q)+\ell_E(p,q)&\ge\frac12\pos{1-z_0(p)-z_0(q)}^2.
 \label{eq:2k2-edge}
\end{align}
If $pq>0$, equality in~\eqref{eq:2k2-nonedge} is possible only when $p=q=1$.
\end{lemma}
\begin{proof}
Let $h_0(p)=3p-p^2$ and $z_*(p)=p\sqrt{h_0(p)/2}$.
The arithmetic--geometric mean inequality gives
\[
 \ell_N(p,q)=\frac{pq}{4}\bigl(h_0(p)+h_0(q)\bigr)
 \ge\frac{pq}{2}\sqrt{h_0(p)h_0(q)}=z_*(p)z_*(q).
\]
For $t=1-p$, we have
\[
 z_0(p)=p^{3/2}(1+t/5),\qquad
 z_*(p)=p^{3/2}\sqrt{1+t/2},
\]
and
\[
 1+\frac t2-\left(1+\frac t5\right)^2
 =\frac{t(5-2t)}{50}\ge0.
\]
Thus $z_0\le z_*$, with strict inequality for $0<p<1$.
This proves~\eqref{eq:2k2-nonedge} and its equality assertion.

For~\eqref{eq:2k2-edge}, write $x=\sqrt p$, $y=\sqrt q$, and
\[
 g=(1-x^2)(1-y^2)-xy,\qquad D=1-z_0(p)-z_0(q).
\]
Suppose first that $g\le0$, so that $\ell_E(p,q)=0$. Define
\[
 T(p)=\frac{\sqrt p}{1-p+\sqrt p}.
\]
This function converts the condition $g\le0$ into an additive bound:
\[
 T(p)+T(q)-1
 =\frac{xy-(1-p)(1-q)}{(1-p+x)(1-q+y)}\ge0.
\]
It therefore suffices to prove $z_0(p)+\sqrt{v(p)}\ge T(p)$.
At $p=0$ this is immediate. For $p>0$, using
$\sqrt{1-p/2}\ge1-p/2$ gives
\[
 \frac{z_0(p)+\sqrt{v(p)}}{x}
 =\frac{p(6-p)}5+(1-p)\sqrt{1-p/2}
 \ge1-\frac3{10}p(1-p).
\]
Also,
\begin{align*}
 &\left(1-\frac3{10}x^2(1-x^2)\right)(1+x-x^2)-1\\
 &\qquad=x(1-x)\left(1-\frac3{10}x(1+x)(1+x-x^2)\right)\ge0,
\end{align*}
since $x\le1$, $1+x\le2$, and $1+x-x^2\le5/4$.
This proves the one-variable bound.
It follows that $D_+\le\sqrt{v(p)}+\sqrt{v(q)}$, and hence
$D_+^2\le2v(p)+2v(q)$, as required.

If $g\ge0$, then $\ell_E(p,q)=g^2$.
Lemma~\ref{lem:app-2k2} proves that the polynomial
\[
 Q_2(x,y)=2v(x^2)+2v(y^2)+2g^2
 -\left(1-\frac{x^3(6-x^2)}5-\frac{y^3(6-y^2)}5\right)^2
\]
is nonnegative on $[0,1]^2$. Therefore
$2v(p)+2v(q)+2\ell_E(p,q)\ge D^2\ge D_+^2$.
\end{proof}

\begin{proposition}\label{prop:2k2}
For every $(p_1,p_2,p_3,p_4)\in[0,1]^4$,
$\CC_{2K_2}(p_1,p_2,p_3,p_4)>1/2$.
\end{proposition}
\begin{proof}
Take $E(2K_2)=\{12,34\}$, and put
$Z=z_0(p_1)+z_0(p_2)$ and $W=z_0(p_3)+z_0(p_4)$.
By Lemma~\ref{lem:2k2-pair},
\begin{equation}\label{eq:2k2-F}
 \CC_{2K_2}\ge F(Z,W):=\frac12(1-Z)_+^2+\frac12(1-W)_+^2+ZW.
\end{equation}
For $0\le Z,W\le1$, the right-hand side equals
$1/2+(Z+W-1)^2/2$. If $Z\ge1\ge W$, it is at least
$(1-W)^2/2+W=1/2+W^2/2$; the symmetric case is the same.
If $Z,W\ge1$, it is at least $ZW\ge1$. Thus $F(Z,W)\ge1/2$.

It remains to exclude equality. If $Z=0$, then $p_1=p_2=0$ and
$\ell_E(p_1,p_2)=1$, so $\CC_{2K_2}\ge1$. The case $W=0$ is identical.
Suppose therefore that $Z,W>0$.
If one of the four cross-pair inequalities~\eqref{eq:2k2-nonedge} is strict,
then~\eqref{eq:2k2-F} is strict and the result follows.
Otherwise, choose a positive coordinate in each of $\{p_1,p_2\}$ and
$\{p_3,p_4\}$. The equality assertion of Lemma~\ref{lem:2k2-pair}
forces both coordinates to equal one. Since $z_0(1)=1$, we have
$ZW\ge1$ and again $\CC_{2K_2}\ge1$. This proves strictness in all cases.
\end{proof}

\subsection{The \texorpdfstring{$C_5$}{C5} case}

Put $f(p)=p(2-p)=1-(1-p)^2$, so that $v(p)=f(p)(1-f(p))/2$.
This gives complementary targets on edges and nonedges. Each vertex
of $C_5$ belongs to four unordered pairs, so summing the following
pair inequalities will give $4\CC_{C_5}$ on the left.

\begin{lemma}\label{lem:c5-pair}
For every $p,q\in[0,1]$,
\begin{align}
 4\ell_E(p,q)+v(p)+v(q)&\ge\frac76(1-f(p))(1-f(q)),\label{eq:c5-edge}\\
 4\ell_N(p,q)+v(p)+v(q)&\ge\frac76f(p)f(q).\label{eq:c5-nonedge}
\end{align}
\end{lemma}
\begin{proof}
For the nonedge inequality, put $s=p+q$ and $t=pq$.
Using the same polynomial $v$ on $[0,2]$, direct expansion gives
\[
 4\ell_N(p,q)+v(p)+v(q)-\frac76f(p)f(q)
 =v(s)+t\left(s^2-\frac23s+\frac13-\frac t6\right).
\]
Here $v(s)\ge0$, and $t\le s^2/4$ implies
\[
 s^2-\frac23s+\frac13-\frac t6
 \ge\frac{23s^2-16s+8}{24}
 =\frac{23(s-8/23)^2+120/23}{24}>0.
\]
This proves~\eqref{eq:c5-nonedge}.

For the edge inequality, put
\[
 \begin{gathered}
 x=1-p,\quad y=1-q,\quad A=xy,\quad r=\sqrt{(1-x)(1-y)},\\
 S=x^2+y^2,\qquad T=x+y,\qquad D=1-S.
 \end{gathered}
\]
Then $\ell_E=(A-r)_+^2$ and $v(p)+v(q)=A^2+SD/2$.
We will show that
\begin{equation}\label{eq:c5-edge-reduced}
 4\ell_E+\frac12SD\ge\frac7{36}A^2.
\end{equation}
If $T\le1$, then $D=1-T^2+2A\ge2A$ and $S\ge2A$, giving
$SD/2\ge2A^2$. Suppose $T\ge1$. We have
\[
 A^2+D=(1-x^2)(1-y^2)=r^2(1+T+A)\ge2r^2,
 \qquad A\le\sqrt{\ell_E}+r.
\]
For every $c>0$, the Cauchy--Schwarz inequality gives
\begin{equation}\label{eq:c5-cauchy}
 4\ell_E+cr^2\ge\frac{4c}{4+c}(\sqrt{\ell_E}+r)^2
 \ge\frac{4c}{4+c}A^2.
\end{equation}
If $D\ge0$, then $S\ge T^2/2\ge1/2$, so
\begin{align*}
 4\ell_E+\frac12SD
 &\ge4\ell_E+\frac14D
 \ge4\ell_E+\frac12r^2-\frac14A^2\\
 &\ge\left(\frac49-\frac14\right)A^2=\frac7{36}A^2.
\end{align*}
If $D<0$, then $S\le2$, and~\eqref{eq:c5-cauchy} with $c=2$ gives
\[
 4\ell_E+\frac12SD\ge4\ell_E+D
 \ge4\ell_E+2r^2-A^2\ge\frac13A^2.
\]
Thus~\eqref{eq:c5-edge-reduced} holds.
Adding $A^2$ proves~\eqref{eq:c5-edge}, since
$43/36\ge7/6$ and $A^2=(1-f(p))(1-f(q))$.
\end{proof}

\begin{proposition}\label{prop:c5}
For every $(p_1,\ldots,p_5)\in[0,1]^5$,
$\CC_{C_5}(p_1,\ldots,p_5)\ge7/12$.
\end{proposition}
\begin{proof}
Write $f_i=f(p_i)$. Sum~\eqref{eq:c5-edge} over the five cycle edges
and~\eqref{eq:c5-nonedge} over the five nonedges.
Each vertex occurs in four pairs. Hence $4\CC_{C_5}\ge\frac76\Phi(f)$, where
\[
 \Phi(f)=\sum_{ij\in E(C_5)}(1-f_i)(1-f_j)
 +\sum_{ij\notin E(C_5)}f_if_j
 =5-2\sum_i f_i+\sum_{i<j}f_if_j.
\]
The function $\Phi$ is affine in each variable separately, so it attains
its minimum on $[0,1]^5$ at a vertex. If exactly $j$ coordinates equal one,
its value is $5-2j+\binom j2$. For $j=0,1,2,3,4,5$, these values are
$5,3,2,2,3,5$. Therefore $\Phi\ge2$, which gives the assertion.
\end{proof}

\subsection{The \texorpdfstring{$C_4$}{C4} case}

For the $C_4$ case, we use the sharper lower bounds $E$ and $N$
from~\eqref{eq:strong-scalar}. Each vertex of $C_4$ belongs to two
edges and one nonedge. We introduce an edge correction $h$ and a
nonedge correction $\gamma$ with $2h+\gamma=4w/5$, where $w=2v$.
After summing the pair inequalities, the remaining vertex term is
$v-2h-\gamma=-3w/10$. Define
\begin{equation}\label{eq:c4-functions}
 \begin{aligned}
 f(p)&=p(2-p),\qquad w(p)=f(p)(1-f(p))=2v(p),\\
 h(p)&=w(p)(1-p)\left(\frac7{20}-\frac{3p}{25}\right),\\
 \gamma(p)&=\frac45w(p)-2h(p)
 =\frac{w(p)}{50}(5+47p-12p^2).
 \end{aligned}
\end{equation}
These functions are nonnegative on $[0,1]$. The identity $w=f-f^2$
makes the remaining vertex terms cancel the individual $f(p_i)^2$
terms in the pair sum, leaving a quadratic in $\sum_i f(p_i)$.

\begin{lemma}\label{lem:c4-pair}
For all $p,q\in[0,1]$,
\begin{align}
 E(p,q)+h(p)+h(q)+\frac1{200}
 &\ge\frac35(1-f(p))(1-f(q)),\label{eq:c4-edge}\\
 N(p,q)+\gamma(p)+\gamma(q)
 &\ge\frac35 f(p)f(q).\label{eq:c4-nonedge}
\end{align}
\end{lemma}
\begin{proof}
Lemma~\ref{lem:app-c4-nonedge} proves~\eqref{eq:c4-nonedge}.
For the edge inequality, put
\[
 A=(1-p)(1-q),\quad g=A-\sqrt{pq},\quad
 H=h(p)+h(q)+\frac1{200}-\frac35A^2.
\]
Lemma~\ref{lem:app-c4-polynomial} proves that
\begin{equation}\label{eq:c4-polynomial}
 (1-pq)g^2+AH\ge0
\end{equation}
for every $p,q\in[0,1]$.
If $g\ge0$ and $A>0$, divide by $A$ to obtain $E(p,q)+H\ge0$.
If $A=0$, the desired inequality is immediate.
In particular, $H\ge0$ on the boundary $g=0$.
When $g<0$, the polynomial bound alone does not suffice: $E$ contains
the positive part of $g$ and vanishes there, so we need $H\ge0$.
Lemma~\ref{lem:app-c4-zero} proves this by a one-variable concavity
argument on the region $g\le0$, using the boundary bound above.
This proves~\eqref{eq:c4-edge}.
\end{proof}

\begin{proposition}\label{prop:c4}
For every $(p_1,p_2,p_3,p_4)\in[0,1]^4$,
\[
 \widetilde\CC_{C_4}\ge\frac3{10}
 \left(\sum_{i=1}^4 f(p_i)-\frac52\right)^2+\frac{101}{200}.
\]
\end{proposition}
\begin{proof}
Write $f_i=f(p_i)$, $w_i=w(p_i)$, and $S=\sum_i f_i$.
Summing Lemma~\ref{lem:c4-pair}, with two edge corrections and one
nonedge correction at each vertex, gives
\[
 \widetilde\CC_{C_4}\ge\frac35\Phi(f)
 +\sum_i\bigl(v(p_i)-2h(p_i)-\gamma(p_i)\bigr)-\frac1{50},
\]
where $1/50$ is the sum of the four edge allowances $1/200$, and
\[
 \Phi(f)=\sum_{ij\in E(C_4)}(1-f_i)(1-f_j)
 +\sum_{ij\notin E(C_4)}f_if_j
 =4-2S+\sum_{i<j}f_if_j.
\]
Since $v-2h-\gamma=-3w/10$, we obtain
\begin{align*}
 \widetilde\CC_{C_4}
 &\ge\frac35\left(\Phi(f)-\frac12\sum_i w_i\right)-\frac1{50}\\
 &=\frac35\left(4-\frac52S+\frac12S^2\right)-\frac1{50}
 =\frac3{10}(S-5/2)^2+\frac{101}{200}.
\end{align*}
\end{proof}

Propositions~\ref{prop:2k2}, \ref{prop:c5}, and~\ref{prop:c4}
prove Theorem~\ref{thm:local}.

Combining the three cases gives the following consequence.

\begin{corollary}\label{cor:nonsplit}
For every non-split graph $G$ and every $1\le k\le |V(G)|$,
\[
 S_k^+(G)<|E(G)|+\binom{k+1}{2}.
\]
\end{corollary}
\begin{proof}
Fix a rank-$k$ orthogonal projection $P$. By
Lemma~\ref{lem:split-characterization}, $G$ contains an induced subgraph
$H$ isomorphic to $2K_2$, $C_4$, or $C_5$. Retaining only the vertex
and pair terms belonging to $H$, Lemma~\ref{lem:scalar} and
Theorem~\ref{thm:local} give
\[
 \sum_{i\in V(G)}v(p_i)+\TT_G(P)>\frac12.
\]
Hence~\eqref{eq:roadmap} yields $\EE_G(P)<0$. Applying
Lemma~\ref{lem:variational} to a maximizing projection proves the assertion.
\end{proof}

\begin{remark}\label{rem:nonsplit-gap}
The preceding argument uses the full projection $P$; its principal
submatrix indexed by $V(H)$ need not be a projection.

Since $\CC_{2K_2}-1/2$ is continuous and positive on $[0,1]^4$,
it has a positive minimum $\delta_2$. Together with
Theorem~\ref{thm:local}, this gives an absolute constant
$\delta=\min\{\delta_2,1/200\}>0$ such that
\[
 S_k^+(G)\le |E(G)|+\binom{k+1}{2}-\delta
\]
for every non-split graph $G$ and every $1\le k\le |V(G)|$.
\end{remark}

\section{Conclusions and remarks}\label{sec:conclusion}

\begin{proof}[Proof of Theorem~\ref{thm:main}]
Theorem~\ref{thm:split} and Lemma~\ref{lem:variational} prove the bound
for split graphs, while Corollary~\ref{cor:nonsplit} gives strictness
for non-split graphs. Equality can therefore occur only for a split graph,
and Theorem~\ref{thm:split-equality} gives precisely the stated equality cases.
\end{proof}

Thus Conjecture~\ref{conj:signless} holds for every graph $G$ and every
$1\le k\le |V(G)|$. Equality holds if and only if $k=1$ and deleting
all isolated vertices from $G$ leaves a star $K_{1,a}$, $a\ge1$,
or a triangle $K_3$.

For graphs $G$ and $H$ on disjoint vertex sets, their \emph{join}
$G\vee H$ is obtained by adding all edges between the two vertex sets.
We write $\overline{K_b}$ for the edgeless graph on $b$ vertices.

Although the bound is strict for $k\ge2$, the gap can tend to zero
with $k$ fixed.

\begin{proposition}\label{prop:asymptotic-sharpness}
Fix an integer $k\ge2$. For an integer $b\ge2$, let
\[
 G_{k,b}=K_k\vee\overline{K_b},
\]
the graph obtained by joining every vertex of a $k$-clique to every
vertex of an independent set of size $b$. Put $T=b+3k-2$.
Then
\begin{equation}\label{eq:join-gap}
 \begin{split}
 |E(G_{k,b})|+\binom{k+1}{2}-S_k^+(G_{k,b})
   &=\frac{T-\sqrt{T^2-8k(k-1)}}2\\
   &=\frac{4k(k-1)}{T+\sqrt{T^2-8k(k-1)}}>0.
 \end{split}
\end{equation}
For each fixed $k\ge2$, this gap is asymptotic to $2k(k-1)/b$,
and hence tends to zero as $b\to\infty$.
In particular,
\[
 \inf_{\substack{G\text{ split}\\|V(G)|\ge k}}
 \left(|E(G)|+\binom{k+1}{2}-S_k^+(G)\right)=0.
\]
Thus no positive constant depending only on $k$ can be subtracted
from the split bound.
\end{proposition}
\begin{proof}
The graph has $kb+\binom k2$ edges, and its signless Laplacian is
\[
 Q(G_{k,b})=
 \begin{pmatrix}
 (b+k-2)I_k+J_k&\one_k\one_b^{\mathsf T}\\
 \one_b\one_k^{\mathsf T}&kI_b
 \end{pmatrix},
\]
where $\one_d$ is the all-ones vector in $\R^d$.
Vectors supported on the clique and having coordinate sum zero
have eigenvalue $b+k-2$, with multiplicity $k-1$.
Vectors supported on the independent set and having coordinate sum
zero have eigenvalue $k$, with multiplicity $b-1$.
The remaining two-dimensional invariant subspace consists of vectors
constant on each part. In the basis of the two part-indicator vectors,
the restriction is represented by
\[
 \begin{pmatrix}b+2k-2&b\\k&k\end{pmatrix}.
\]
Its characteristic polynomial is
\[
 g(\lambda)=(\lambda-b-2k+2)(\lambda-k)-bk
             =\lambda^2-T\lambda+2k(k-1).
\]
The roots are
\[
 \lambda_\pm=\frac{T\pm\sqrt{T^2-8k(k-1)}}2.
\]
They are real and distinct, since the discriminant can also be
written as $(b+k-2)^2+4bk>0$.
Their product and sum are positive, so both roots are positive.
Moreover, $g(k)=-bk<0$, and hence $\lambda_-<k<\lambda_+$.
Also $g(b+2k-2)=-bk<0$, so $\lambda_+>b+2k-2>b+k-2$.
Since $b\ge2$, we have $k\le b+k-2$.
It follows that the sum of the largest $k$ eigenvalues is
\[
 S_k^+(G_{k,b})=\lambda_++(k-1)(b+k-2).
\]
When $b=2$ the middle eigenvalues may coincide, but the same sum
still holds.

Using $\binom k2+\binom{k+1}{2}=k^2$, direct subtraction gives
\begin{align*}
 |E(G_{k,b})|+\binom{k+1}{2}-S_k^+(G_{k,b})
  &=kb+k^2-(k-1)(b+k-2)-\lambda_+\\
  &=b+3k-2-\lambda_+=T-\lambda_+=\lambda_-.
\end{align*}
Rationalizing the numerator proves~\eqref{eq:join-gap}.
For fixed $k$, $T/b\to1$ and
$\sqrt{T^2-8k(k-1)}/b\to1$, so
\[
 b\lambda_-\longrightarrow2k(k-1).
\]
By Theorem~\ref{thm:split}, the gap is nonnegative for every split graph.
For each fixed $k\ge2$, the gap for $G_{k,b}$ tends to zero as
$b\to\infty$, proving the stated infimum and showing that the upper
bound is asymptotically sharp even within the class of split graphs.
\end{proof}

Therefore, we conclude that Conjecture~\ref{conj:signless} holds for all
simple undirected graphs and that its upper bound is asymptotically sharp for each fixed
$k\ge2$. Theorem~\ref{thm:main} gives a complete characterization of
the equality cases.

The local estimates distinguish non-split graphs from this family.
Remark~\ref{rem:nonsplit-gap} gives a positive absolute gap for
non-split graphs, whereas the split gaps approach zero at each fixed
rank $k\ge2$.

\begin{remark}
The maximum principle of \citet{Fan} reduces
the problem to bounding $\tr(PQ(G))$
over rank-$k$ orthogonal projections $P$.
For the ordinary Laplacian, Lew~\citep{Lew,LewPartition} obtained
approximate bounds for Brouwer's conjecture.
\citet{KT} proved the conjecture, and \citet*{CCYZ}
characterized its equality cases.
Our proof is motivated by the projection formulation of
\citet[Section~5]{KT}. The argument of Kothari and Tudose uses
$P\one=0$, where $\one$ is the
all-ones vector. This relation is not available in general for the
signless Laplacian. We therefore derive
two identities suited to the signless setting and treat split
and non-split graphs separately.

For split graphs, we use a block-matrix argument based on the
clique--independent-set partition. The identity in
Lemma~\ref{lem:block} separates two squared Frobenius norms and a
degree term from $\tr(PQ(G))$. Rank and inverse estimates then give
the required bound, and the equality conditions yield the stars
and triangles in Theorem~\ref{thm:main}.

For non-split graphs, the key identity is
Lemma~\ref{lem:identity}. It writes the excess
$\tr(PQ(G))-m-\binom{k+1}{2}$ as $1/2$ minus a sum of nonnegative
terms. Every non-split graph contains an induced $2K_2$, $C_4$, or
$C_5$. We show that the vertex and pair terms belonging to any such
induced subgraph already sum to more than $1/2$, so the excess is
negative. These local inequalities follow from two-variable
estimates and explicit nonnegative polynomial expansions.
\end{remark}

\par\bigskip
\appendix
\pdfbookmark[0]{Appendices: Polynomial identities and equality conditions}{appendices}
\section*{Appendices: Polynomial identities and equality conditions}
\section{A polynomial inequality for two disjoint edges}\label{app:two-edges}

For an integer $m\geq0$, define the normalized Bernstein polynomials
\[
 \beta_{m,j}(t)=\binom mj t^j(1-t)^{m-j},\qquad 0\leq j\leq m.
\]
They are nonnegative and have sum one on $[0,1]$. For a list
$a=(a_0,\ldots,a_m)$, write
\begin{equation}\label{app:bernstein}
 \mathcal B[a](t)=\sum_{j=0}^m a_j\beta_{m,j}(t).
\end{equation}
A polynomial with nonnegative Bernstein coefficients is nonnegative on the
unit interval. The bivariate version follows in the same way.

\begin{lemma}\label{lem:app-2k2}
For $x,y\in[0,1]$, put $g=(1-x^2)(1-y^2)-xy$ and
\[
 Q_2(x,y)=2v(x^2)+2v(y^2)+2g^2
 -\left(1-\frac{x^3(6-x^2)}5-\frac{y^3(6-y^2)}5\right)^2.
\]
Then $Q_2(x,y)\geq0$.
\end{lemma}

\begin{proof}
We separate an explicit square and express the residual in a basis
whose elements are nonnegative on $[0,1]^2$. Use the unnormalized
Bernstein basis polynomials $b_i(t)=t^i(1-t)^{10-i}$, for
$0\leq i\leq10$. Thus $b_i=\beta_{10,i}/\binom{10}{i}$; this
normalization allows the coefficients below to be listed as integers.
The identity is
\begin{equation}\label{app:2k2-identity}
 100Q_2(x,y)=25(x+y-1)^2(x-y)^2(4+3xy)^2
 +\sum_{i,j=0}^{10}A_{ij}b_i(x)b_j(y),
\end{equation}
The entries of the symmetric integer matrix $A$ are determined from
the power coefficients of the residual by~\eqref{app:2k2-coefficients}.
The following table lists them; their nonnegativity gives the required
sign in~\eqref{app:2k2-identity}.
\begin{center}
\small
\setlength{\tabcolsep}{5pt}
\begin{tabular}{r|rrrrrr}
\toprule
$i\backslash j$&0&1&2&3&4&5\\
\midrule
0&100&1000&3900&8240&10780&9200\\
1&1000&10400&41800&90200&118800&100000\\
2&3900&41800&169500&363000&465875&373500\\
3&8240&90200&363000&754562&918084&679230\\
4&10780&118800&465875&918084&1031288&698860\\
5&9200&100000&373500&679230&698860&504250\\
6&4956&51960&178470&288280&299130&386312\\
7&1424&14040&40180&52830&104140&326898\\
8&124&1040&555&1812&52964&201148\\
9&0&0&0&3274&22468&64150\\
10&0&0&0&200&1475&4500\\
\bottomrule
\end{tabular}

\medskip
\begin{tabular}{r|rrrrr}
\toprule
$i\backslash j$&6&7&8&9&10\\
\midrule
0&4956&1424&124&0&0\\
1&51960&14040&1040&0&0\\
2&178470&40180&555&0&0\\
3&288280&52830&1812&3274&200\\
4&299130&104140&52964&22468&1475\\
5&386312&326898&201148&64150&4500\\
6&585720&596960&346070&99080&7406\\
7&596960&594210&327260&89590&7124\\
8&346070&327260&174412&47004&3959\\
9&99080&89590&47004&12698&1120\\
10&7406&7124&3959&1120&100\\
\bottomrule
\end{tabular}
\end{center}
Every entry of $A$ is nonnegative. Thus each term on the right of
\eqref{app:2k2-identity} is nonnegative on $[0,1]^2$. For an explicit
coefficient check, write
\[
 100Q_2(x,y)-25(x+y-1)^2(x-y)^2(4+3xy)^2
 =\sum_{k,l=0}^{10}c_{kl}x^ky^l.
\]
Then the entries in the table satisfy
\begin{equation}\label{app:2k2-coefficients}
 A_{ij}=\sum_{k=0}^i\sum_{l=0}^j
 c_{kl}\binom{10-k}{i-k}\binom{10-l}{j-l}.
\end{equation}
Indeed, $t^k=\sum_{i=k}^{10}\binom{10-k}{i-k}b_i(t)$ by the binomial
theorem. Applying this identity in both variables proves
\eqref{app:2k2-identity} from \eqref{app:2k2-coefficients}.
\end{proof}

\section{A nonedge inequality for the four-cycle}\label{app:four-cycle-nonedge}

We use $f,w,h,\gamma$ from \eqref{eq:c4-functions} and the Bernstein
notation \eqref{app:bernstein}.

\begin{lemma}\label{lem:app-c4-nonedge}
For every $p,q\in[0,1]$,
\[
 N(p,q)+\gamma(p)+\gamma(q)\geq\frac35 f(p)f(q).
\]
\end{lemma}

\begin{proof}
Set
\[
 F(p,q)=50\left\{pq(p+q-2pq)
 -(1-pq)\left(\frac35 f(p)f(q)-\gamma(p)-\gamma(q)\right)\right\}
\]
and
\[
\begin{aligned}
 S(p,q)=9(p-q)^2\bigl[&(1-p)^2q^2\{3p-2(1-q)\}^2\\
 &+(1-q)^2p^2\{3q-2(1-p)\}^2\bigr].
\end{aligned}
\]
For $pq<1$, the left side minus the right side of the desired inequality
is $F(p,q)/[50(1-pq)]$. Since $S\geq0$, it suffices to prove
$R:=F-S\geq0$. Write $z=(p+q)/2$ and $d=(p-q)/2$.
Then $0\leq z\leq1$ and $|d|\leq\min(z,1-z)$.
Symmetry makes $R(z+d,z-d)$ a polynomial in $d^2$. We normalize
$d^2$ by $z^2$ for $z\leq1/2$ and by $(1-z)^2$ for $z\geq1/2$,
with the endpoints treated separately.

The Bernstein coefficients in these normalized variables are given
in~\eqref{app:nonedge-lower} and~\eqref{app:nonedge-upper} below.
They use $A_i=\mathcal B[\alpha_i]$ and $B_i=\mathcal B[\beta_i]$,
whose coefficient lists are as follows.
\begin{center}
\small
\renewcommand{\arraystretch}{1.65}
\begin{tabular}{cl}
\toprule
$\alpha_1$&$\left(40,\ \frac{725}{14},\ \frac{3967}{84},\ \frac{937}{28},\ \frac{1359}{70},\ \frac{1541}{168},\ \frac{1707}{448},\ \frac{181}{64}\right)$\\
$\alpha_2$&$\left(60,\ \frac{183}{2},\ \frac{2349}{28},\ \frac{15559}{280},\ \frac{547}{20},\ \frac{531}{56},\ \frac{307}{112},\ \frac{15}{8}\right)$\\
$\alpha_3$&$\left(40,\ \frac{983}{12},\ \frac{5429}{60},\ \frac{2871}{40},\ \frac{121}{3},\ \frac{305}{24},\ 4\right)$\\
$\alpha_4$&$\left(5,\ \frac{119}{6},\ \frac{137}{6},\ 2\right)$\\
$\beta_1$&$\left(\frac{181}{64},\ \frac{1249}{448},\ \frac{325}{96},\ \frac{1279}{280},\ \frac{1287}{140},\ \frac{905}{42},\ \frac{615}{14},\ 80\right)$\\
$\beta_2$&$\left(\frac{15}{8},\ \frac{17}{4},\ \frac{501}{56},\ \frac{7337}{560},\ \frac{3007}{140},\ \frac{555}{14},\ \frac{975}{14},\ 120\right)$\\
$\beta_3$&$\left(11,\ \frac{251}{12},\ \frac{1511}{60},\ \frac{2621}{80},\ \frac{263}{6},\ \frac{685}{12},\ 80\right)$\\
$\beta_4$&$\left(\frac92,\ \frac{19}{4},\ 5,\ \frac{35}{8},\ 5\right)$\\
\bottomrule
\end{tabular}
\end{center}
All entries in these lists are positive. Define
\begin{equation}\label{app:T}
 T(z)=12z^6-83z^5+133z^4+49z^3-120z^2+19z+10.
\end{equation}
For $0<z\leq1/2$, put $t=d^2/z^2$. Expansion in the degree-four Bernstein
basis gives
\begin{equation}\label{app:nonedge-lower}
 R(z+d,z-d)=\sum_{j=0}^4a_j(z)\beta_{4,j}(t),
\end{equation}
where
\[
\begin{aligned}
 a_0&=2z(1-z)T(z),& a_1&=\frac z2 A_1(2z),& a_2&=\frac z3 A_2(2z),\\
 a_3&=\frac{z(1-2z)}2A_3(2z),&
 a_4&=4z(1-2z)^2A_4(2z).
\end{aligned}
\]
For $1/2\leq z<1$, instead put $t=d^2/(1-z)^2$. The expansion is
\begin{equation}\label{app:nonedge-upper}
 R(z+d,z-d)=\sum_{j=0}^4b_j(z)\beta_{4,j}(t),
\end{equation}
where
\[
\begin{aligned}
 b_0&=2z(1-z)T(z),& b_1&=\frac{1-z}{2}B_1(2z-1),&
 b_2&=\frac{1-z}{3}B_2(2z-1),\\
 b_3&=\frac{(1-z)(2z-1)}2B_3(2z-1),&
 b_4&=8(1-z)(2z-1)^2B_4(2z-1).
\end{aligned}
\]
Both identities follow by substituting the displayed coefficient lists and
collecting powers of $d^2$. All factors are nonnegative on the respective
half-intervals, except that the sign of $T$ remains to be checked.

Put $y=z-5/8$, so that $-5/8\leq y\leq3/8$. Expansion gives
\[
 T(z)=12y^6-38y^5-\frac{897}{16}y^4+\frac{927}{8}y^3
 +\frac{111025}{1024}y^2-\frac{313}{2048}y+\frac{3725}{65536}.
\]
On this interval, $y^3\geq-5y^2/8$, $y^4\leq25y^2/64$, and
$y^5\leq27y^2/512$. Therefore
\[
\begin{aligned}
 T(z)&\geq\left(\frac{111025}{1024}-\frac{4635}{64}
 -\frac{22425}{1024}-\frac{513}{256}\right)y^2
 -\frac{313}{2048}y+\frac{3725}{65536}\\
 &=\frac{3097}{256}y^2-\frac{313}{2048}y+\frac{3725}{65536}\\
 &\geq11y^2+\left(y-\frac{313}{4096}\right)^2
 +\frac{855631}{16777216}>0.
\end{aligned}
\]
Thus \eqref{app:nonedge-lower} and \eqref{app:nonedge-upper} prove $R\geq0$.
At $z=0$ or $z=1$, we have $d=0$ and $R=0$ directly. It follows that
$F=S+R\geq0$ everywhere. For $pq<1$, divide by $50(1-pq)$ to obtain
the desired inequality. At $p=q=1$, it reads $1\geq3/5$.
\end{proof}

\section{An edge inequality for the four-cycle}\label{app:four-cycle-edge}

We first prove a polynomial inequality on the whole square. It gives
\eqref{eq:c4-edge} on the region where the positive part in $E(p,q)$ is
active. A concavity argument then treats the remaining region.

\begin{lemma}\label{lem:app-c4-polynomial}
For $x,y\in[0,1]$, let $p=x^2$, $q=y^2$, and put
\[
 A=(1-p)(1-q),\quad u=pq,\quad g=A-xy,\quad
 H=h(p)+h(q)+\frac1{200}-\frac35A^2.
\]
Then the polynomial
\begin{equation}\label{app:edge-P}
 P(x,y)=(1-u)g^2+AH
\end{equation}
is nonnegative on $[0,1]^2$.
\end{lemma}

\begin{proof}
We separate the nonnegative square $(p-q)^2g^2$ and prove nonnegativity
of the remaining polynomial through its Bernstein representations. Write
\begin{equation}\label{app:edge-R}
 R(x,y)=P(x,y)-(p-q)^2g^2
 =\bigl[1-pq-(p-q)^2\bigr]g^2+AH.
\end{equation}
It suffices to prove $R\geq0$. Set $m=(x+y)/2$ and $d=(x-y)/2$, so
$|d|\leq\min(m,1-m)$. Symmetry in $x,y$ gives an even expansion
\begin{equation}\label{app:edge-power}
 R(m+d,m-d)=\sum_{j=0}^8c_j(m)d^{2j}.
\end{equation}
The polynomials $c_j$ are the coefficients obtained by expanding
\eqref{app:edge-R} after this substitution.
The maximum admissible value of $d^2$ changes from $m^2$ to $(1-m)^2$
at $m=1/2$. We therefore normalize $d^2$ separately on these two halves
and then treat the resulting Bernstein coefficients as polynomials in
one variable.

\medskip\noindent\textit{The lower half, $0<m\leq1/2$.}
Put $t=d^2/m^2$ and $v=4m^2$, both in $[0,1]$. Conversion to the degree-eight
Bernstein basis yields
\begin{equation}\label{app:edge-lower}
 R(m+d,m-d)=\sum_{i=0}^8 B_i^-(v)\beta_{8,i}(t),
\end{equation}
where
\begin{equation}\label{app:edge-lower-coeff}
 B_i^-(4m^2)=\sum_{j=0}^i\frac{\binom ij}{\binom8j}c_j(m)m^{2j}.
\end{equation}
Formula~\eqref{app:edge-lower-coeff} converts the power coefficients
$c_j(m)$ into the Bernstein coefficients $B_i^-(v)$ in the variable $t$.
The table gives a second Bernstein representation, now in $v$, for each
of these functions. For $0\leq i\leq7$, its meaning is
\[
 B_i^-(v)=\frac1{D_i}\sum_{j=0}^{n_i}A_{ij}\beta_{n_i,j}(v),
\]
and the last row means
\[
 B_8^-(v)=\frac{1-v}{6000}\sum_{j=0}^6 A_{8j}\beta_{6,j}(v).
\]
\begin{center}
\small
\setlength{\tabcolsep}{4pt}
\begin{tabular}{crrl}
\toprule
$i$&$n_i$&$D_i$&$(A_{i0},\ldots,A_{in_i})$\\
\midrule
0&8&45875200&\begin{tabular}[c]{@{}l@{}}$(18579456,\ 17131520,\ 14948864,\ 12442880,\ 9907200,$\\$\quad7542336,\ 5475800,\ 3778537,\ 2478336)$\end{tabular}\\[4pt]
1&8&1835008000&\begin{tabular}[c]{@{}l@{}}$(743178240,\ 706693120,\ 609745920,\ 487788800,\ 364035072,$\\$\quad252767680,\ 161778480,\ 94170895,\ 49698040)$\end{tabular}\\[4pt]
2&8&3211264000&\begin{tabular}[c]{@{}l@{}}$(1300561920,\ 1274219520,\ 1083805440,\ 832055040,\ 581647872,$\\$\quad367606240,\ 205835260,\ 98996905,\ 40557720)$\end{tabular}\\[4pt]
3&8&802816000&\begin{tabular}[c]{@{}l@{}}$(325140480,\ 327931520,\ 274167680,\ 201697440,\ 131949184,$\\$\quad75674900,\ 36685940,\ 14224735,\ 4428200)$\end{tabular}\\[4pt]
4&8&100352000&\begin{tabular}[c]{@{}l@{}}$(40642560,\ 42163520,\ 34551600,\ 24323400,\ 14873248,$\\$\quad7703568,\ 3180338,\ 965069,\ 238448)$\end{tabular}\\[4pt]
5&8&10035200&\begin{tabular}[c]{@{}l@{}}$(4064256,\ 4333560,\ 3471084,\ 2335145,\ 1332588,$\\$\quad618470,\ 210944,\ 48342,\ 11424)$\end{tabular}\\[4pt]
6&8&12544000&\begin{tabular}[c]{@{}l@{}}$(5080320,\ 5563460,\ 4343585,\ 2789050,\ 1482386,$\\$\quad608380,\ 160755,\ 27055,\ 7980)$\end{tabular}\\[4pt]
7&7&336000&\begin{tabular}[c]{@{}l@{}}$(136080,\ 155355,\ 102970,\ 51909,\ 18252,$\\$\quad2755,\ 330,\ 105)$\end{tabular}\\[4pt]
8&6&6000&\begin{tabular}[c]{@{}l@{}}$(2430,\ 3330,\ 2514,\ 1479,\ 590,$\\$\quad30,\ 30)$\end{tabular}\\
\bottomrule
\end{tabular}
\end{center}
The table identities can be checked by substituting $v=4m^2$ and
comparing coefficients in $m$ with~\eqref{app:edge-lower-coeff}.
Every listed integer is positive, and
$1-v\geq0$. Thus $B_i^-(v)\geq0$ for every $i$, proving $R\geq0$ on
the lower half. Continuity includes $m=0$.

\medskip\noindent\textit{The upper half, $1/2\leq m<1$.}
Put $r=2m-1$ and $t=d^2/(1-m)^2$. Again $r,t\in[0,1]$, and
\begin{equation}\label{app:edge-upper}
 R(m+d,m-d)=\sum_{i=0}^8 B_i^+(r)\beta_{8,i}(t),
\end{equation}
where
\begin{equation}\label{app:edge-upper-coeff}
 B_i^+(r)=\sum_{j=0}^i\frac{\binom ij}{\binom8j}
 c_j\left(\frac{1+r}{2}\right)\left(\frac{1-r}{2}\right)^{2j}.
\end{equation}
We prove that all nine coefficients are nonnegative. For $0\leq i\leq7$,
they factor as
\begin{equation}\label{app:edge-upper-factor}
 B_i^+(r)=\kappa_i(r)P_i(r),\qquad
 P_i(r)=a_i+b_ir+e_ir^2+r^3G_i(r),
\end{equation}
with the following factors and coefficients.
\begin{center}
\small
\renewcommand{\arraystretch}{1.6}
\begin{tabular}{ccrrr}
\toprule
$i$&$\kappa_i(r)$&$a_i$&$b_i$&$e_i$\\
\midrule
0&$\frac{(1-r)(r+3)}{819200}$&14752&-113994&131715\\
1&$\frac{1-r}{6553600}$&177493&-1657593&3190878\\
2&$\frac{1-r}{11468800}$&144849&-1868423&6597426\\
3&$\frac{1-r}{2867200}$&15815&-299220&1706788\\
4&$\frac{1-r}{1792000}$&4258&-113326&987502\\
5&$\frac{1-r}{179200}$&204&-5908&81596\\
6&$\frac{1-r}{89600}$&57&-1043&28967\\
7&$\frac{1-r}{3200}$&1&0&539\\
\bottomrule
\end{tabular}
\end{center}
For the remaining polynomials, partial sums of the power coefficients
give a useful lower bound. If $G(r)=\sum_{j=0}^n g_jr^j$ and
$s_j=\sum_{\ell=0}^j g_\ell$, then $G=\mathcal A[s]$, where
\begin{equation}\label{app:partial-sums}
 \mathcal A[s](r)=(1-r)\sum_{j=0}^{n-1}s_jr^j+s_nr^n.
\end{equation}
The weights in this expression are nonnegative and sum to one on $[0,1]$.
Hence $\mathcal A[s](r)\geq\min_j s_j$. The remainders $G_i$ for
$1\leq i\leq7$ satisfy $G_i=\mathcal A[\sigma_i]$ for the lists below.
\begin{center}
\small
\begin{tabular}{cl}
\toprule
$\sigma_1$&\begin{tabular}[c]{@{}l@{}}$(5148306,\ 12644441,\ 16576334,\ 12792506,\ 10014974,$\\$\quad10880785,\ 11640466,\ 11528336,\ 11389634,\ 11381875,$\\$\quad11393302,\ 11396206,\ 11396422)$\end{tabular}\\[4pt]
$\sigma_2$&\begin{tabular}[c]{@{}l@{}}$(3958762,\ 13173549,\ 27548320,\ 24172072,\ 15105136,$\\$\quad15789303,\ 18612806,\ 18640464,\ 18094354,\ 18006679,$\\$\quad18048124,\ 18062488,\ 18063748)$\end{tabular}\\[4pt]
$\sigma_3$&\begin{tabular}[c]{@{}l@{}}$(376774,\ 1682156,\ 6224315,\ 6695595,\ 3754395,$\\$\quad3332162,\ 4381720,\ 4562664,\ 4345194,\ 4287246,$\\$\quad4303205,\ 4310345,\ 4311017)$\end{tabular}\\[4pt]
$\sigma_4$&\begin{tabular}[c]{@{}l@{}}$(81094,\ 522369,\ 3420268,\ 4479568,\ 2651124,$\\$\quad1903484,\ 2633476,\ 2923786,\ 2755352,\ 2684659,$\\$\quad2697550,\ 2704918,\ 2705566)$\end{tabular}\\[4pt]
$\sigma_5$&\begin{tabular}[c]{@{}l@{}}$(3184,\ 43659,\ 305434,\ 452710,\ 311674,$\\$\quad199910,\ 259242,\ 306812,\ 291218,\ 280005,$\\$\quad281356,\ 282436,\ 282508)$\end{tabular}\\[4pt]
$\sigma_6$&\begin{tabular}[c]{@{}l@{}}$(715,\ 34893,\ 146047,\ 215967,\ 176879,$\\$\quad118955,\ 133511,\ 161483,\ 157103,\ 150041,$\\$\quad150499,\ 151195,\ 151219)$\end{tabular}\\[4pt]
$\sigma_7$&\begin{tabular}[c]{@{}l@{}}$(0,\ 2122,\ 5522,\ 7058,\ 6658,$\\$\quad5294,\ 5294,\ 6050,\ 6050,\ 5836,$\\$\quad5836,\ 5860)$\end{tabular}\\
\bottomrule
\end{tabular}
\end{center}
These lists give the partial sums of the power coefficients of $G_i$.
In particular, $G_i\geq0$ for $1\leq i\leq7$,
and $G_1\geq5148306$. For $2\leq i\leq7$, $a_i>0$, $e_i>0$, and the
values of $4a_ie_i-b_i^2$ are, respectively,
\[
 331517727767,\quad18438800480,\quad3976351788,\quad
 31677872,\quad5516627,\quad2156.
\]
Thus $a_i+b_ir+e_ir^2>0$, which proves $P_i>0$ for these indices.

The remaining polynomial $G_0$ has the degree-eleven Bernstein expansion
\[
 G_0(r)=\sum_{j=0}^{11}\rho_j\beta_{11,j}(r),
\]
where
\[
\begin{aligned}
 (\rho_0,\ldots,\rho_{11})=\Bigl(&388624,\ \frac{4467478}{11},\ \frac{23181626}{55},\ \frac{71609161}{165},\\
 &\frac{73096544}{165},\ \frac{103473386}{231},\ \frac{103557439}{231},\ \frac{146447809}{330},\\
 &\frac{71650456}{165},\ \frac{23086566}{55},\ \frac{4406238}{11},\ 377127\Bigr).
\end{aligned}
\]
Every coefficient is at least $377127$, so $G_0\geq377127$. Consequently,
\[
\begin{aligned}
 P_0(r)&\geq14752-113994r+131715r^2+377127r^3,\\
 P_1(r)&\geq177493-1657593r+3190878r^2+5148306r^3.
\end{aligned}
\]
For the first bound, put $t=r-2/9$. Since $t\geq-2/9$, we have
$t^3\geq-(2/9)t^2$. Expansion therefore gives
\[
 P_0(r)\geq299327t^2+\frac{1250}{3}t+\frac{5104}{81}>0,
\]
because
\[
 4(299327)\frac{5104}{81}-\left(\frac{1250}{3}\right)^2
 =\frac{6096997532}{81}>0.
\]
For the second, use $t=r-9/50$ and $t^3\geq-(9/50)t^2$ to obtain
\[
 P_1(r)\geq\frac{126106704}{25}t^2-\frac{10576971}{1250}t
 +\frac{783476737}{62500}>0.
\]
This quadratic $at^2+bt+c$ has $a>0$ and
$4ac-b^2=395094803539444551/1562500>0$, so it is positive. Finally,
\[
 B_8^+(r)=r^4(1-r)(1+r)\geq0.
\]
All factors $\kappa_i(r)$ are nonnegative, so every $B_i^+$ is nonnegative.
Equation \eqref{app:edge-upper} proves $R\geq0$ on the upper half,
including $m=1$ by continuity. Restoring the subtracted square in
\eqref{app:edge-R} proves the lemma.
\end{proof}

\begin{lemma}\label{lem:app-c4-zero}
Let $A=(1-p)(1-q)$, $g=A-\sqrt{pq}$, and
\[
 H(p,q)=h(p)+h(q)+\frac1{200}-\frac35A^2.
\]
If $H\geq0$ whenever $g=0$, then $H\geq0$ throughout the region
$g\leq0$ in $[0,1]^2$.
\end{lemma}

\begin{proof}
Put
\[
 z=\frac{p+q}{2},\qquad\delta=\frac{(p-q)^2}{4},\qquad
 c_0=\frac{3-\sqrt5}{2}.
\]
The constraints $p,q\in[0,1]$ are equivalent to $0\leq z\leq1$ and
$0\leq\delta\leq\min\{z^2,(1-z)^2\}$. Moreover,
\[
 pq=z^2-\delta,\qquad A=(1-z)^2-\delta.
\]
If $z<c_0$, let $s=\sqrt{pq}\leq z<1/2$. Since $s^2-s$ is decreasing
on $[0,1/2]$,
\[
 g=1-2z+s^2-s\geq1-3z+z^2>0.
\]
Thus $g\leq0$ requires $z\geq c_0$.

For fixed $z\geq c_0$, define
\[
 F_z(\delta)=z^2-\delta-\bigl((1-z)^2-\delta\bigr)^2.
\]
Since $A\geq0$ on the admissible interval, $g\leq0$ is equivalent to
$F_z(\delta)\geq0$. We have $F_z(0)\geq0$ and
\[
 F_z'(\delta)=2\bigl((1-z)^2-\delta\bigr)-1
 \leq2(1-c_0)^2-1<0.
\]
Call an interval inactive when $g\leq0$, so that $E(p,q)=0$ there.
If $c_0\leq z\leq1/2$, then $F_z(z^2)=-(1-2z)^2\leq0$. The inactive
interval is therefore $[0,\delta_*]$, where $F_z(\delta_*)=0$, or
equivalently $g=0$. If $z\geq1/2$, then
$F_z((1-z)^2)=2z-1\geq0$; the entire admissible interval
$[0,(1-z)^2]$ is inactive, and one coordinate equals one at its right
endpoint.

Set $H_0(z,\delta)=h(p)+h(q)-3A^2/5$. Direct expansion gives
\begin{equation}\label{app:edge-concavity}
\begin{aligned}
 100H_0(z,\delta)={}&-24\delta^3-2(180z^2-475z+313)\delta^2\\
 &-2(z-1)(180z^3-770z^2+868z-209)\delta\\
 &-2(z-1)^3(12z^3-59z^2+100z-30).
\end{aligned}
\end{equation}
The polynomial $A_0(z)=180z^2-475z+313$ decreases on $[0,1]$ and is at
least $A_0(1)=18$. Hence
\[
 \frac{\partial^2}{\partial\delta^2}(100H_0)
 =-144\delta-4A_0(z)<0.
\]
Thus $H_0+1/200$ is concave in $\delta$, and it suffices to check the
endpoints of each inactive interval.

At $\delta=0$,
\[
 H_0(z,0)=\frac{(1-z)^3}{50}C(z),\qquad
 C(z)=12z^3-59z^2+100z-30.
\]
The derivative $C'(z)=36z^2-118z+100$ decreases on $[0,1]$ and is at
least $C'(1)=18$. Also $c_0^2=3c_0-1$ gives
$C(c_0)=19c_0-7>0$, since $c_0>3/8$. Therefore $H_0(z,0)\geq0$ for
$z\geq c_0$. At the other endpoint, the boundary hypothesis applies when
$z\leq1/2$. When $z\geq1/2$, one coordinate equals one, so $A=0$ and
$H_0$ is $h$ of the other coordinate, which is nonnegative. Both endpoints
have $H_0+1/200\geq0$. Concavity completes the proof.
\end{proof}

\section{Equality in the auxiliary estimates}\label{app:auxiliary-equality}\label{subsec:intermediate-equality}
We use the notation of Section~\ref{sec:split} and the two cases in
the proof of Theorem~\ref{thm:split}. For the vanishing conditions
on $L_H,L_Y$ and $Z_\mu$, see~\eqref{eq:equality-zero-squares} and
the proof of Theorem~\ref{thm:split-equality}.

\paragraph{The inverse trace estimates.}
In Case~1, put $u_i=r_i^{-1}$, $\sigma=\sum_i u_i$, and
$t_{ih}=(BB^{\mathsf T})_{ih}$. The diagonal of
$BB^{\mathsf T}+I_c-M$ is zero and its off-diagonal entries are
$t_{ih}-1$. Formula~\eqref{eq:inverse} therefore gives
\begin{equation}\label{eq:equality-universal-trace}
 c-\kappa=\frac{2}{1+\sigma}
                  \sum_{i<h}u_i u_h(t_{ih}-1).
\end{equation}
Every $t_{ih}\ge1$, and every coefficient is positive.
Consequently, $\kappa=c$ if and only if every two distinct clique
vertices have exactly one common independent neighbor. For $c=1$
the condition is vacuous. Sharpness of this trace estimate alone
allows private leaves; the degree loss and the coupling square
exclude those leaves when equality in the whole argument is required
and $c\ge2$.

In Case~2, the inverse formula gives
\begin{equation}\label{eq:auxiliary-proper-kappa}
 \kappa=c-\frac{\sigma}{1+\sigma}
 -\frac{2}{1+\sigma}\sum_{i<h}u_i u_h(BB^{\mathsf T})_{ih}.
\end{equation}
Thus $\kappa=c-\sigma/(1+\sigma)$ if and only if $BB^{\mathsf T}$ is
diagonal. This is equivalent to every independent vertex having
degree at most one: an off-diagonal entry counts independent vertices
adjacent to two specified clique vertices. Since all $r_i\ge1$,
sharpness forces $B$ to have full row rank. Its incompatibility with
$L_Y=0$ at $k=c-1$ is the contradiction in Case~2 of the proof of
Theorem~\ref{thm:split-equality}.

\begin{proposition}\label{prop:scalar-equality}
Under the hypotheses of Lemma~\ref{lem:scalarM},
\[
 \phi(\theta)=\frac1{1+\sigma}
 \quad\Longleftrightarrow\quad c=2\ \text{and}\ d_1=d_2=2.
\]
\end{proposition}
\begin{proof}
If $\sigma<1$, the proof of Lemma~\ref{lem:scalarM} gives
\[
 \phi(\theta)\ge\frac1{\sigma(1+\sigma)}
                   >\frac1{1+\sigma}.
\]
If $\sigma>1$, it instead gives
$\phi(\theta)\ge1/2>1/(1+\sigma)$.
Thus equality requires $\sigma=1$.
Since $\theta\ge2$ and $\phi$ is strictly increasing on $[1,\infty)$,
equality then forces $\theta=2$.

Choose a nonzero vector $z\in\Ker(M-2I_c)$. Then
\[
 0=z^{\mathsf T}(M-2I_c)z
    =\sum_i(d_i-2)z_i^2+\left(\sum_i z_i\right)^2.
\]
Every summand is nonnegative. Nonzero coordinates of $z$ can occur
only at indices with $d_i=2$, and their sum must be zero.
There are therefore at least two such indices.
They already contribute $1/2+1/2=1$ to $\sigma$, so there can be no
further index. Hence $c=2$ and $d_1=d_2=2$.
Conversely, $M=2I_2+J_2$ has $\theta=2$ and $\sigma=1$,
which gives equality.
\end{proof}

In Case~2, Proposition~\ref{prop:scalar-equality} forces
$c=2$ and $r_1=r_2=1$. Since no independent vertex is adjacent to
both clique vertices, the two clique vertices have distinct private
neighbors. The graph is therefore $P_4$ with possible isolated
vertices. However,
\[
 \det(\lambda I_4-Q(P_4))
       =\lambda(\lambda-2)(\lambda^2-4\lambda+2),
\]
so $q_1(P_4)=2+\sqrt2<4=m+1$; equality in the scalar estimate
does not give equality in the spectral bound.

\begin{proposition}\label{prop:rank-equality}
Under the hypotheses of Lemma~\ref{lem:rankloss}, set
\[
 \mathcal K=\Ker X,\qquad d=\dim\mathcal K,\qquad
 \rho=(c-k)_+,\qquad F=M^{1/2}H-M^{-1/2},
\]
and let $\Pi$ be the orthogonal projection onto $\mathcal K$.
With $\phi(M)=M-2I_c+M^{-1}$, the exact identity is
\begin{equation}\label{eq:rank-equality-defect}
 \begin{split}
 L_H-\rho\phi(\theta)
   ={}&\|F(I_c-\Pi)\|_F^2\\
     &+\tr\!\left(\Pi[\phi(M)-\phi(\theta)I_c]\right)
       +(d-\rho)\phi(\theta).
 \end{split}
\end{equation}
All three terms on the right are nonnegative.
Thus equality in~\eqref{eq:rankloss} holds exactly when all three vanish.
If $\theta>1$, this is equivalent to the simultaneous conditions
\[
 d=\rho,\qquad
 \mathcal K\subseteq\Ker(M-\theta I_c),\qquad
 F\big|_{\mathcal K^\perp}=0.
\]
The last condition means $H=M^{-1}$ on $\mathcal K^\perp$;
$H$ is already the identity on $\mathcal K$.
If $\theta=1$, equality is simply $L_H=0$, or $H=M^{-1}$.
\end{proposition}
\begin{proof}
Orthogonal decomposition of the domain of $F$ gives
\[
 \|F\|_F^2=\|F\Pi\|_F^2+\|F(I_c-\Pi)\|_F^2.
\]
Since $H\Pi=\Pi$,
$F\Pi=(M^{1/2}-M^{-1/2})\Pi$.
It follows, using $\Pi^2=\Pi$ and cyclicity of trace, that
\[
 \|F\Pi\|_F^2=\tr(\Pi\phi(M)\Pi)=\tr(\Pi\phi(M)).
\]
Subtract $\rho\phi(\theta)$ and use $\tr\Pi=d$ to obtain
\eqref{eq:rank-equality-defect}.
The first term is a squared norm. The second is nonnegative because
$\phi(M)-\phi(\theta)I_c\succeq0$, by monotonicity of $\phi$ on
$[1,\infty)$. The last is nonnegative because
$d\ge\rho$ and $\phi(\theta)\ge0$.

If $\theta>1$, then $\phi(\theta)>0$, so the last term vanishes
exactly when $d=\rho$.
For an orthonormal basis $z_1,\ldots,z_d$ of $\mathcal K$, the second
term is the sum of the nonnegative numbers
$z_j^{\mathsf T}[\phi(M)-\phi(\theta)I_c]z_j$.
It vanishes exactly when every $z_j$ lies in the kernel of that
positive semidefinite matrix. Strict monotonicity of $\phi$ makes
this kernel $\Ker(M-\theta I_c)$.
The first term vanishes exactly when $F$ is zero on
$\mathcal K^\perp$, proving the stated equality conditions.
Finally, if $\theta=1$, the right side of the rank bound is zero,
so its equality condition is $L_H=0$, equivalent to $H=M^{-1}$.
\end{proof}

\section*{Declarations}

\noindent\textbf{Conflict of interest.} The authors declare that they have no known competing financial interests or
personal relationships that could have appeared to influence the work reported in this paper.

\noindent\textbf{Data availability.} Data sharing not applicable to this paper as no datasets were generated or analysed
during the current study.

\noindent\textbf{Use of generative AI tools.} The authors used GPT-6 Astra to assist in exploring proof strategies,
checking proofs, and improving the exposition. All mathematical results, arguments, and proofs were
independently reviewed and verified by the authors, who take full responsibility for the accuracy and
content of this work.

 \section*{Acknowledgment}

The work is partly supported by the National Natural Science Foundation of China (Nos. 12371354, W2521102), the Science and Technology Commission of Shanghai Municipality (No. 25LN3200600) and the Montenegrin-Chinese Science and Technology Cooperation Project (No. 4-3).


\end{document}